\documentclass[reqno, a4paper, 11pt]{article}
\usepackage{amsfonts, amsmath, amssymb, amsthm, color}
\usepackage[hmargin=3cm, vmargin=2.6cm]{geometry}
\usepackage[varg]{txfonts}
\usepackage{hyperref}

\newtheorem{thm}{Theorem}[section]
\newtheorem{lemma}[thm]{Lemma}
\newtheorem{prop}[thm]{Proposition}
\newtheorem{cor}[thm]{Corollary}

\theoremstyle{definition}
\newtheorem{rmk}[thm]{Remark}

\newcommand{\ep}{\epsilon}
\newcommand{\mr}{\mathbb{R}}
\newcommand{\mh}{\mathcal{H}}
\newcommand{\mw}{\mathcal{W}}
\newcommand{\mz}{\mathcal{Z}}

\newcommand{\ds}{{\delta,\sigma}}
\newcommand{\dsz}{\delta, \sigma_0}
\newcommand{\dsy}{{\delta,\sigma(y)}}
\newcommand{\dsn}{{\delta_n,\sigma_n}}
\newcommand{\ls}{{\lambda, \sigma}}
\newcommand{\els}{{\ep^{\alpha}\lambda, \sigma}}
\newcommand{\elep}{{\ep^{\alpha}\lambda_{\ep}, \sigma_{\ep}}}

\newcommand{\hx}{H^1(X;\rho^{1-2s})}
\newcommand{\dmr}{D^1(\mr^{N+1}_+;t^{1-2s})}
\newcommand{\bg}{\bar{g}}
\newcommand{\hh}{\hat{h}}
\newcommand{\hs}{\hat{\sigma}}
\newcommand{\tf}{\tilde{f}}
\newcommand{\ww}{\widehat{W}}
\newcommand{\wu}{\widetilde{U}}
\newcommand{\whp}{\widehat{\Phi}}
\newcommand{\wtp}{\widetilde{\Phi}}
\newcommand{\pa}{\partial}

\newcommand{\la}{\left \langle}
\newcommand{\ra}{\right \rangle}

\renewcommand{\(}{\left(}
\renewcommand{\)}{\right)}

\begin{document}
\title{On perturbations of the fractional Yamabe problem}

\author{Woocheol Choi, Seunghyeok Kim}

\newcommand{\Addresses}{{
\bigskip \footnotesize

\noindent (Woocheol Choi) \textsc{Department of Mathematical Sciences, Seoul National University}\par\nopagebreak
\noindent \textit{E-mail address}: \texttt{chwc1987@math.snu.ac.kr}

\medskip
\noindent (Seunghyeok Kim) \textsc{Facultad de Matem\'{a}ticas, Pontificia Universidad Cat\'{o}lica de Chile, Avenida Vicu\~{n}a Mackenna 4860, Santiago, Chile}\par\nopagebreak
\noindent \textit{E-mail address}: \texttt{shkim0401@gmail.com}
}}

\date{\today}
\maketitle

\begin{abstract}
The fractional Yamabe problem, proposed by Gonz\'{a}lez-Qing (2013) \cite{CG}, is a geometric question which concerns the existence of metrics with constant fractional scalar curvature.
It extends the phenomena which were discovered in the classical Yamabe problem and the boundary Yamabe problem to the realm of nonlocal conformally invariant operators.
We investigate a non-compactness property of the fractional Yamabe problem by constructing bubbling solutions to its small perturbations.
\end{abstract}



\allowdisplaybreaks
\tableofcontents

\section{Introduction}
Suppose that $(X^{N+1},g^+)$ is an asymptotically hyperbolic (A.H.) manifold with the conformal infinity $(M^N,[\hh])$
and $P_{\hh}^s = P^s[g^+, \hh]$ is the fractional Paneitz operator with the principal symbol $(-\Delta_{\hh})^s$.
We are concerned with two low order perturbations of the fractional Yamabe equation
\begin{equation}\label{eq_main_1}\tag_{$1_{\pm}$}
P_{\hh}^s u + f u = u^{{N+2s \over N-2s} \pm \ep} \quad \text{on } (M, \hh), \qquad u > 0 \quad \text{on } (M, \hh),
\end{equation}
and
\setcounter{equation}{1}
\begin{equation}\label{eq_main_2}
P_{\hh}^s u + \ep f u = u^{N+2s \over N-2s} \quad \text{on } (M, \hh), \qquad u > 0 \quad \text{on } (M, \hh)
\end{equation}
where $f$ is a $C^1$-function on $M$, $\ep > 0$ is a small parameter and $s \in (0,1)$.
(Equations $(1_+)$ and $(1_-)$ correspond to the supercritical and subcritical problem, respectively.)
As one can observe, \eqref{eq_main_1} is a manifold analogue of the fractional Lane-Emden-Fowler equation with a slightly subcritical or supercritical exponent,
while \eqref{eq_main_2} can be viewed as a version of the fractional Brezis-Nirenberg problem on A.H. manifolds.

\medskip
For $s \in (0,1)$, Gonz\'{a}lez-Qing \cite{GQ} and Gonz\'{a}lez-Wang \cite{GW} studied the fractional Yamabe problem
which is a geometric problem to find a metric $\hh_0$ in the conformal class $[\hh]$ of $\hh$ with the constant fractional scalar curvature $Q_{\hh_0}^s = P_{\hh_0}^s(1)$.
The existence of such a metric follows from a solution of the non-local equation
\begin{equation}\label{eq_yamabe}
P_{\hh}^s u = c u^{N+2s \over N-2s} \quad \text{on } (M, \hh), \qquad u > 0 \quad \text{on } (M, \hh)
\end{equation}
with some $c \in \mr$.
As in the classical Yamabe problem, the sign of $c$ depends on that of the fractional Yamabe invariant
\begin{equation}\label{eq_yamabe_inv}
\mu_{\hh}^s(M) = \inf_{h \in [\hh]} {\int_M Q_h^s dv_h \over \(\int_M dv_h\)^{N-2s \over N}} = \inf_{\substack{u \in C^{\infty}(M),\\ u > 0}} {\int_M uP_{\hh}^su dv_{\hh} \over \(\int_M u^{2N \over N-2s} dv_{\hh}\)^{N-2s \over N}},
\end{equation}
and the fractional Yamabe problem is solvable if the inequality \[-\infty < \mu_{\hh}^s(M) < \mu_{h_c}^s\(S^N\)\]
holds where the manifold $(S^N, h_c)$ is the $N$-dimensional unit sphere with the canonical metric as the boundary of the Poincar\'e ball.
In \cite{GQ,GW}, it is shown that the above inequality is valid for A.H. manifolds with non-umbilic boundary or the non-locally conformally flat A.H. manifolds with umbilic boundary under some additional dimensional and technical assumptions.

Since the operator $P_{\hh}^s = P^s[g^+, \hh]$ reduces to the conformal Laplacian if $s = 1$ and $(X,g^+)$ is Poincar\'e-Einstein  (refer to \eqref{eq_paneitz_1}),
the fractional Yamabe problem can be understood as a direct generalization of the classical one towards the non-local conformally invariant operators.
This fact being one of the reasons, recently intensive studies on the fractional conformal operators have been conducted by lots of researchers.
We refer \cite{GMS, JX, ACH, CLZ, JLX, JLX2, JLX3, QR} and references therein where closely related problems to ours, e.g., the singular fractional Yamabe problem,
the fractional Yamabe flow and the fractional Nirenberg problem are investigated.

\medskip
After the classical Yamabe problem is completely solved by the contribution of Yamabe, Trudinger, Aubin and Schoen \cite{Ya, Tr, Au, Sc}, Schoen proposed a question on the compactness of its solution set.
Remarkably, it turned out that the solution set is indeed compact in the $C^2$-topology provided that the dimension of the background manifold is at most 24 \cite{KMS},
but it may be false for some manifolds whose dimension is greater than or equal to 25 \cite{Br, BM}.

Furthermore, as a low order perturbation, equations \eqref{eq_main_1} and \eqref{eq_main_2} in the local case $s = 1$
have been in the limelight (see \cite{LZ, Dr, Dr2, DH, DH2, MPV, PVe, EPV} among other possible references).
It was revealed that these equations also have an interesting feature.
In particular, if the operator $P^1_{\hh} + f$ is coercive, then the solution set of $(1_-)$ should be compact when $N \ge 3$ and $f < 0$ in $M$ (\cite{Dr2}),
but non-compact in the case that $N \ge 4$ and there is a region of $M$ where $f > 0$ (\cite{MPV, PVe}).

The main objective of this paper is to extend previous results regarding the compactness or stability property of conformal operators to the nonlocal setting $s \in (0,1)$ by considering \eqref{eq_main_1} and \eqref{eq_main_2}.
As a result, a perturbation of the boundary Yamabe problem (corresponding to the case $s = 1/2$) is partly covered here as a byproduct of our main results in the case of \eqref{eq_main_1}.
For the existence results of the boundary Yamabe problem in the Euclidean case and in the setting of compact Riemannian manifolds, see Adimurthi-Yadava \cite{AY}, Escobar \cite{Es} and Marques \cite{Ma}.
We also should mention that equations with $s = 2$ (see \eqref{eq_paneitz_2}) were investigated in Deng-Pistoia \cite{DP} and Pistoia-Vaira \cite{PV}.

\medskip
For the existence of solutions to the fractional Yamabe problem,
equivalent minimization problems to \eqref{eq_yamabe_inv} which only contain local differential operators can be derived
by exploiting the extension theorem of Chang and Gonz\'{a}lez \cite{CG}.
The authors of \cite{GQ, GW} utilized this observation to deduce the existence result, instead finding a minimizer that attains the Yamabe invariant $\mu^s_{\hh}(M)$ in a direct manner.
After the fundamental extension result of Caffarelli and Silvestre \cite{CS} for the fractional Laplacians on $\mr^N$, such a standpoint,
introducing and studying equivalent extended local problems rather than considering nonlocal problems itself, has been highlighted by many researchers.
See for example \cite{CT, BCPS, ST, CS2, DDW, CKL} and references therein.
In this paper, we keep on use this strategy.

According to \cite{CG} (see Proposition \ref{prop_ext} below), it is natural to consider the following degenerate equation with the weighted Neumann boundary condition
\begin{equation}\label{eq_Dirichlet}
-\text{div}\(\rho^{1-2s} \nabla U\) + E(\rho)U = 0 \quad \text{in } (X, \bg) \quad \text{and} \quad \pa_{\nu}^sU = 0 \quad \text{on } (M, \hh)
\end{equation}
where
\begin{equation}\label{eq_weighted_normal}
\pa_{\nu}^s U := -\kappa_s \cdot \lim_{\rho \to 0+} \rho^{1-2s} {\pa U \over \pa \rho} \quad \text{with } \kappa_s:= {\Gamma(s) \over 2^{1-2s} \Gamma(1-s)}
\end{equation}
($\nu$ is the outward normal vector to $M = \pa X$)
in order to understand equations with the fractional Paneitz operator $P^s_{\hh}$.
Let $H$ be the trace of the second fundamental form $\pi$ of $(M, \hh)$ as the boundary of $(X, \bg)$ and $\hx$ the weighted Sobolev space whose precise definition is given in Section \ref{sec_scheme}.
Our paper deals with the situation when the first eigenvalue of \eqref{eq_Dirichlet} is positive (modulo the effect of the function $\tf$ to be introduced below),
that is, there exists a constant $C > 0$ such that the inequality
\begin{equation}\label{eq_coercive}
\int_X \(\rho^{1-2s} |\nabla U|^2_{\bg} + E(\rho) U^2\)dv_{\bg} + \int_M \tf U^2 dv_{\hh} \ge C \int_X \rho^{1-2s} U^2 dv_{\bg} \end{equation}
holds for arbitrary functions $U \in \hx$,
where the function $\tf$ on $M$ is defined to be
\[\tf = \begin{cases}
f &\text{if \eqref{eq_main_1} is considered},\\
0 & \text{if \eqref{eq_main_2} is considered}.
\end{cases}\]

Under the coercivity assumption \eqref{eq_coercive}, we have the following non-compactness result for \eqref{eq_main_1}.
Recall that for any $C^1$ function $\psi$ on $M$, a critical point $x_0 \in M$ is called to be $C^1$-stable if there is a small neighborhood $\Lambda$ of $x_0$ in $M$ such that
$\nabla \psi(x) = 0$ for some $x \in \overline{\Lambda}$ implies $x = x_0$ and deg$(\nabla\psi, \Lambda, 0) \ne 0$ (see \cite{Li2}).
Here deg denotes the Brouwer degree.
It is well-known that any isolated local minimum point and maximum point is a $C^1$-stable critical point. Moreover, so is a nondegenerate critical point if $\psi$ is a $C^2$-function.
\begin{thm}\label{thm_main_1}
Suppose that $s \in (0,1)$, $N > \max\{4s,1\}$ and $H = 0$ if $s \in [1/2,1)$.
Assume also that \eqref{eq_coercive} is true.
\begin{enumerate}
\item If the function $f$ possesses a $C^1$-stable critical point $\sigma_0 \in M$ such that $f(\sigma_0) > 0$,
then for sufficiently small $\ep > 0$ equation $(1_+)$ admits a positive solution $u_{\ep} \in C^{1,\beta}(M)$ which blows up at $\sigma_0$ as $\ep \to 0$.
\item If the function $f$ possesses a $C^1$-stable critical point $\sigma_0 \in M$ such that $f(\sigma_0) < 0$,
then for sufficiently small $\ep > 0$ equation $(1_-)$ admits a positive solution $u_{\ep} \in C^{1,\beta}(M)$ which blows up at $\sigma_0$ as $\ep \to 0$.
\end{enumerate}
Here the H\"{o}lder exponent $\beta \in (0,1)$ is determined by $N$ and $s$.
\end{thm}

\noindent Furthermore, we can obtain an existence theorem for \eqref{eq_main_2} where the geometric object $H$ plays an important role.
\begin{thm}\label{thm_main_2}
Suppose that $s \in (0,1/2)$, $N \ge 2$, 
as well as \eqref{eq_coercive} hold.
Also, let $\lambda: M \to [0, \infty]$ be a function defined as
\[\lambda(\sigma) = \begin{cases}
\(\dfrac{-4N(1-2s)sd_1f(\sigma)}{\(2N(N-1)+\(1-4s^2\)\)d_s H(\sigma)}\)^{1 \over 1-2s} &\text{if } H(\sigma) \ne 0 \text{ and } \dfrac{f(\sigma)}{H(\sigma)} \in (-\infty, 0],\\
\infty &\text{otherwise}
\end{cases} \]
where the positive numbers $d_1$ and $d_s$ are given in Subsection \ref{subsec_energy}.
If $(\lambda_0, \sigma_0) := (\lambda(\sigma_0), \sigma_0)$ is a $C^1$-stable critical point of the function
\[\widetilde{J}(\ls) = {d_1 \over 2}f(\sigma)\lambda^{2s} + \({2N(N-1)+\(1-4s^2\) \over 4N(1-2s)}\) d_s H(\sigma)\lambda \quad \text{for } (\ls) \in (0,\infty) \times M\]
such that $\lambda(\sigma_0) > 0$, then for $\ep > 0$ small enough equation \eqref{eq_main_2} has a positive solution $u_{\ep} \in C^{1, \beta}(M)$ which blows up at $\sigma_0$ as $\ep \to 0$.
Furthermore $\sigma_0$ is necessarily a critical point of the function $|f|/|H|^{2s}$ on $M$.
The exponent $\beta \in (0,1)$ again depends on $N$ and $s$.
\end{thm}

The analogous existence results to ours in the Euclidean setting,
that is, a proof for the existence of solutions for the fractional Lane-Emden-Fowler equation and the Brezis-Nirenberg problem in smooth bounded domains of $\mr^N$ can be found in \cite{CKL, DLS}.
While we are studying here a \textit{small perturbation} of equation \eqref{eq_yamabe} defined on \textit{general} manifolds to understand its non-compactness characteristic,
one may address a \textbf{dual} problem: to construct a \textit{particular} metric for which \textit{original} equation \eqref{eq_yamabe} has the solution set that is not $L^{\infty}$-bounded.
It is investigated in \cite{KMW}, which extends \cite{Br, BM, Al, WZ} to a nonlocal setting.

To deduce our existence result, we shall employ the finite dimensional Lyapunov-Schmidt reduction method.
As far as we know, this paper is the first attempt to apply the reduction procedure towards equations with the fractional Paneitz operators defined in general manifolds.
For applications of the reduction method to the fractional Laplacians in the Euclidean setting or the fractional Paneitz operators under a particular choice of the metric, we refer to \cite{CKL, DDW, KMW} and so on.

Our problems require more delicate computations compared to problems on Euclidean spaces.
The main reason making them harder is that the fractional Paneitz operator $P_{\hh}^s = P^s[g^+, \hh]$ depends not only on the metric $\hh$ on the boundary $M$, but also on the metric $g^+$ in the interior $X$.
In other words, the boundary $M$ does not contain whole information in contrast with problems with fractional Laplacians $(-\Delta)^s$ on the Euclidean spaces,
and so it is inevitable to look carefully how the interior $X$ plays a role in our problem.
This is achieved by inspecting the extended problem given in Proposition \ref{prop_ext}.
To overcome the other difficulties we face, we have to also establish a certain regularity result (Lemma \ref{lemma_trace}),
compute decay of the $s$-harmonic extensions of the bubbles \eqref{eq_bubble} (Lemma \ref{lemma_decay}),
use the weighted Sobolev trace inequality \eqref{eq_trace_mf} for compact manifolds elaborately,
employ the dual characterization of the norm \eqref{eq_sobolev_norm_2} in estimating the error term (Lemma \ref{lemma_error_est}) and others.

\bigskip \noindent \textbf{Notations.}

\noindent - An element of the upper half space $\mr^{N+1}_+$ is denoted by $(x,t)$ where $x \in \mr^N$ and $t > 0$.

\noindent - For any weakly differentiable function $U$ on $\mr^{N+1}_+$, we denote
$\nabla_x U = (\pa_{x_1} U, \cdots, \pa_{x_N} U)$ and $\nabla U = (\nabla_x U, \pa_t U)$.
Also $\pa_{x_i}$ is often written as $\pa_i$.

\noindent - $B_r^+ = B^{N+1}(0,r) \cap \mr^{N+1}_+$ is the $(N+1)$-dimensional half open ball of radius $r$ centered at the origin.

\noindent - $u_+ = \max\{u,0\}$ and $u_- = \max\{-u,0\}$.

\noindent - $\Gamma$ denotes the Gamma function.

\noindent - For any $N \in \mathbb{N}$ and $s \in (0, \min\{1, N/2\})$, we denote $p = {N+2s \over N-2s}$.

\noindent - $C > 0$ is a generic constant, which may change line by line.

\section{Preliminaries}
In this section, we present some geometric and analytical backgrounds to understand our problem.
Most of materials are taken from \cite{CG,GQ,Es,CS,CS2}.

\subsection{Review on conformal fractional Laplacians}\label{subsec_conf_frac}
Let $(X^{N+1}, g^+)$ be an $(N+1)$-dimensional smooth Riemannian manifold with the boundary $M^N$.
We call a function $\rho$ on the closure $\overline{X}$ of $X$ a defining function of the boundary $M$ if $\rho > 0$ in $X$, $\rho = 0$ on $M$ and $d \rho \ne 0$ on $M$.
The manifold $(X, g^+)$ is said to be conformally compact (C.C.) if there is a defining function $\rho$
making $(\overline{X}, \bg)$ be compact where $\bg := \rho^2g^+$.
Also, given the metric $\hh = \bg|_M$, the boundary $(M, [\hh])$ with the conformal class $[\hh]$ of $\hh$ is called the conformal infinity.
A C.C. metric $g^+$ is asymptotically hyperbolic (A.H.) if the sectional curvature approaches to -1 at the infinity $M$,
whose model case is the hyperbolic space:
\[(X, g^+) = (\mathbb{H}^{N+1}, g_{\mathbb{H}}) = \(\mr_+^{N+1}, {|dx|^2 + dt^2 \over t^2}\) \quad \text{or} \quad \(B^{N+1}, {4(|dx|^2 + dt^2) \over (1-|x|^2-t^2)^2}\).\]

According to Graham-Lee \cite{GL}, for an A.H. manifold $X$ and a representative $\hh$ for the conformal class on $(M, [\hh])$, there is a unique special defining function such that
\[g^+ = \rho^{-2}\(d\rho^2 + h_{\rho}\),\ h_{\rho} = \hh + O(\rho)\]
near $M$. It is called the geodesic boundary defining function.

Suppose that \verb"z" $\in \mathbb{C}$, $\text{Re}(\verb"z") > N/2$ and $f \in C^{\infty}(M)$.
Then, by \cite{MM, GZ}, unless $\verb"z"(N-\verb"z")$ is an $L^2$-eigenvalue of $-\Delta_{g^+}$, the following eigenvalue problem
\begin{equation}\label{eq_eigen_1}
\left[-\Delta_{g^+} - \verb"z"(N-\verb"z")\right] V = 0 \text{ in } X
\end{equation}
has a solution of the form
\begin{equation}\label{eq_eigen_2}
V = F\rho^{N-\verb"z"} + G\rho^{\verb"z"},\ F, G \in C^{\infty}(X) \text{ and } F|_{\rho = 0} = f.
\end{equation}
Throughout the paper the existence of such a solution is always assumed.
The scattering operator on $M$ is then defined to be
\[S(\verb"z")f = G|_M,\]
which is a meromorphic family of pseudo-differential operators in $\{\verb"z" \in \mathbb{C}: \text{Re}(\verb"z") > N/2\}$.
In addition, we introduce its normalization so called the fractional Paneitz operator $P_{\hh}^s$, namely
\[P_{\hh}^s = P^s[g^+,\hh] = \begin{cases}
-2^{2s}\dfrac{s\Gamma(s)}{\Gamma(1-s)} S\(\dfrac{N}{2} + s\) &\text{for } s \notin \mathbb{N},\\
(-1)^s2^{2s}s!(s-1)! \cdot \text{Res}_{\verb"z" = N/2 + s} S(\verb"z") &\text{for } s \in \mathbb{N},
\end{cases}\]
whose principal symbol is exactly $(-\Delta_{\hh})^s$.
In the special case that $(X,g^+)$ is Poincar\'e-Einstein (both C.C. and Einstein) and $s = 1$ or 2, we have
\begin{equation}\label{eq_paneitz_1}
P^1_{\hh}u = -\Delta_{\hh}u + {N-2 \over 4(N-1)}R_{\hh}u
\end{equation}
the usual conformal Laplacian, and
\begin{equation}\label{eq_paneitz_2}
P^2_{\hh}u = (-\Delta_{\hh})^2u - \text{div}_{\hh}\(\(\tilde{c}_1R_{\hh}\hh-\tilde{c}_2 \text{Ric}_{\hh}\)du\) + {N-4 \over 2}Q_{\hh}u
\end{equation}
the Paneitz operator.
Here $Q$ stands for the Branson's $Q$-curvature and $\tilde{c}_1, \tilde{c}_2 > 0$ are constants.
The important property of $P^s_{\hh}$ is that it is conformally covariant in the sense that
\[P^s_{\hh u^{4 \over N-2s}}\phi = u^{-{N+2s \over N-2s}}P^s_{\hh}(u\phi) \quad \text{for any function } u > 0 \text{ on } M.\]
Finally, we set the fractional scalar curvature $Q^s_{\hh}$ by $P^s_{\hh}(1)$.

\subsection{Caffarelli-Silvestre's result \cite{CS} and Chang-Gonz\'{a}lez's extension \cite{CG}}
In this subsection, we recall the observation of Chang and Gonz\'{a}lez \cite{CG} which identifies two fractional Laplacians arising in different contexts:
one given as normalized scattering operators \cite{GZ} described above and one originated from the Dirichlet-Neumann operators due to Caffarelli and Silvestre \cite{CS}.

\medskip
For $s \in (0,1)$, let $\dmr$ be the completion of $C_c^{\infty}(\overline{\mr^{N+1}_+})$ with respect to the weighted Sobolev norms
\[\|U\|_{\dmr} := \(\int_{\mr^{N+1}_+} t^{1-2s}|\nabla U(x,t)|^2dxdt\)^{1/2}\]
with the weight $t^{1-2s}$.
Furthermore, we designate by $H^s(\mr^N)$ the standard fractional Sobolev space given as
\[H^s\(\mr^N\) = \left\{u \in L^2\(\mr^N\): \|u\|_{H^s(\mr^N)} := \(\int_{\mr^N} \(1+|\xi|^{2s}\) \left|\hat{u}(\xi)\right|^2 d\xi\)^{1/2} < \infty \right\}\]
where $\hat{u}$ denotes the Fourier transform of $u$,
and define the fractional Laplacian $(-\Delta)^s: H^s(\mr^N) \to H^{-s}(\mr^N)$ to be
\[\widehat{((-\Delta)^s u)}(\xi) = \(2\pi|\xi|\)^{2s} \hat{u}(\xi) \quad \text{for any } \xi \in \mr^N  \text{ given } u \in H^s\(\mr^N\).\]
In the celebrated work of Caffarelli and Silvestre \cite{CS}, the authors found that if $U \in \dmr$ is a unique solution of the equation
\begin{equation}\label{eq_gamma-ext}
\left\{ \begin{array}{ll} \text{div}\(t^{1-2s} \nabla U\) = 0 &~\text{in}~ \mr^{N+1}_+,\\
U(x,0)= u(x) &~ \text{for}~ x \in \mr^N,
\end{array}\right.
\end{equation}
provided a fixed function $u \in H^s(\mr^N)$,
then $(-\Delta)^s u = \pa_{\nu}^s U|_{\mr^N}$ where the definition of the weighted normal derivative $\pa_{\nu}^s$ is given in \eqref{eq_weighted_normal}.
Let us call this $U$ the $s$-harmonic extension of $u$ and denote it by $\text{Ext}^s(u)$.

\medskip
It turned out that this extension result is a special case of the following proposition obtained by Chang and Gonz\'{a}lez \cite{CG}.
We also refer to Section 2 of \cite{GQ}.
\begin{prop} \label{prop_ext} (\cite[Theorem 5.1 and Theorem 4.3]{CG})
Let $(X^{N+1}, g^+)$ be an asymptotically hyperbolic manifold with the conformal infinity $(M^N, [\hh])$ and $\rho$ the geodesic defining function of $\hh$.
Assume also that $H = 0$ if $s \in (1/2, 1)$.
For a smooth function $u$ on $M$, if $V$ is a solution of \eqref{eq_eigen_1} and satisfies \eqref{eq_eigen_2} in which $f$ is substituted with $u$, the function $U := \rho^{\verb"z"-N}V$ solves
\[-\textnormal{div}\(\rho^{1-2s} \nabla U\) + E(\rho)U = 0 \quad \text{in } (X, \bg) \quad \text{and} \quad U = u \quad \text{on } (M, \hh)\]
given that $E(\rho) := \rho^{-1-\verb"z"}(-\Delta_{g^+}-\verb"z"(N-\verb"z"))\rho^{N-\verb"z"}$,
$2\verb"z" := N + 2s$ and $\bg := \rho^2g^+$. Moreover,
\[P^s_{\hh} u = \left\{ \begin{array}{ll}
\pa_{\nu}^s U & \text{for } s \in (0,1) \setminus \{1/2\}, \\
\pa_{\nu}^s U + {N-1 \over 2N} Hu &\text{for } s = 1/2.
\end{array}\right.\]
Here $H$ denotes the trace of the second fundamental form $(\pi_{ij}) = \(-\langle \nabla_{\pa_{\rho}} \pa_i, \pa_j \rangle_{\hh}\)$ on $M = \pa X$ and the operator $\pa_{\nu}^s$ is the weighted normal derivative defined in \eqref{eq_weighted_normal} with $t$ replaced by $\rho$.

For sufficiently small $r_1 > 0$, it also holds that
\begin{equation}\label{eq_E_local}
E(\rho) = {N-2s \over 4N} \left[R_{\bg} \rho^{1-2s} - \(R_{g^+}+N(N+1)\)\rho^{-1-2s}\right] \quad \text{on } M \times (0,r_1).
\end{equation}
\end{prop}

\begin{rmk}
Since it holds that
\[R_{g^+} = - N(N+1) + N\rho \pa_\rho \log(\det h(\rho)) + \rho^2 R_{\bg} \quad \text{on } M \times (0,r_1)\]
and
\[\left. \pa_\rho \log(\det h(\rho))\right|_{\rho = 0} = \left. \text{Tr}\(h(\rho)^{-1} \pa_{\rho} h(\rho)\) \right|_{\rho = 0} = -2H,\]
the remainder term $E(\rho)$ in \eqref{eq_E_local} is reduced to
\begin{equation}\label{eq_E_local_2}
\begin{aligned}
E(\rho)(z) &= -\({N-2s \over 4}\) \pa_\rho \log(\det h(\rho))(\sigma) \rho^{-2s}\\
&= -\({N-2s \over 4}\) \left. \pa_\rho \log(\det h(\rho))\right|_{\rho = 0}(\sigma) \rho^{-2s} + O\(\rho^{1-2s}\)
= \({N-2s \over 2}\) H(\sigma) \rho^{-2s} + O\(\rho^{1-2s}\)
\end{aligned} \end{equation}
for $z = (\sigma, \rho) \in M \times (0,r_1)$.
\end{rmk}

\noindent In particular, our main equation \eqref{eq_main_1} is equivalent to the problem
\begin{equation}\label{eq_main_12}\tag{$15_{\pm}$}
\left\{\begin{array}{ll}
-\text{div}\(\rho^{1-2s} \nabla U\) + E(\rho) U = 0 &\text{in } (X, \bg),\\
\pa_\nu^s U = u^{p \pm \ep} - fu &\text{on } (M, \hh),\\
U = u > 0 &\text{on } (M, \hh)
\end{array}\right.
\end{equation}
and it remains the same as well except the second equation in \eqref{eq_main_12} is replaced by
\setcounter{equation}{15}
\begin{equation}\label{eq_main_21}
\pa_\nu^s U = u^p - \ep f u \quad \text{for } s \in (0,1/2) \quad \text{on } (M, \hh)
\end{equation}
if we deal with \eqref{eq_main_2}.

\medskip
In \cite{CG}, it is also proved that given a geodesic defining function $\rho$, there is another special defining function $\rho^*$ such that $E(\rho^*) = 0$.
\begin{prop} \label{prop_ext_2} (\cite[Theorem 4.7]{CG}, \cite[Proposition 2.2]{GQ})
Assume that $H = 0$ if $s \in (1/2, 1)$.
For a smooth function $u$ on $M$, if $V$ satisfies \eqref{eq_eigen_1} as well as \eqref{eq_eigen_2} in which $f$ is substituted with $u$, the function $U := (\rho^*)^{\verb"z"-N}V$ is a solution of
\begin{equation}\label{eq_ext_1}
-\textnormal{div}\(\(\rho^*\)^{1-2s} \nabla U\) = 0 \quad \text{in } (X, g^*) \quad \text{and} \quad U = u \quad \text{on } (M, \hh)
\end{equation}
where $g^* := (\rho^*)^2g^+$. Moreover $g^*|_M = \hh$, $(\rho^*/\rho)|_M = 1$ and
\begin{equation}\label{eq_ext_2}
P^s_{\hh} u = \pa_{\nu}^s U + Q_{\hh}^s u
\end{equation}
where $Q_{\hh}^s$ is the fractional scalar curvature and
the operator $\pa_{\nu}^s$ is defined in \eqref{eq_weighted_normal} with $t$ substituted with $\rho^*$.
\end{prop}
\noindent This observation is useful in showing a priori $L^{\infty}$-estimate or the strong maximum principle of the operator $P^s_{\hh}$. Refer to \cite[Section 3]{GQ}. (cf. Lemma \ref{lemma_trace} and Proposition \ref{prop_red} below)

\subsection{Sharp trace inequality and its related equations}\label{subsec_trace_ineq}
Given any number $\delta> 0$ and point $\sigma = (\sigma_1, \cdots, \sigma_N) \in \mr^N$, let
\begin{equation}\label{eq_bubble}
w_\ds (x) = \tilde{\kappa}_{N,s} \( \frac{\delta}{\delta^2+|x-\sigma|^2}\)^{N-2s \over 2} \quad \text{for } x \in \mr^N
\quad \text{with } \tilde{\kappa}_{N,s} = 2^{N-2s \over 2} \( \frac{\Gamma \( \frac{N+2s}{2}\)}{\Gamma \( \frac{N-2s}{2}\)}\)^{\frac{N-2s}{4s}}.
\end{equation}
Its constant multiples attain the equality for the sharp Sobolev inequality
\[\( \int_{\mr^N} |u|^{2N \over N-2s} dx \)^{N-2s \over 2N} \le \mathcal{S}_{N,s} \( \int_{\mr^N} \left|(-\Delta)^{s/2} u\right|^2 dx \)^{1 \over 2}\]
where $\mathcal{S}_{N,s} > 0$ is the optimal Sobolev constant, and in particular solve
\begin{equation}\label{eq_entire}
(-\Delta)^s u = u^p, \quad u > 0 \quad \text{in } \mr^N\quad \text{and}\quad \lim_{|x|\to \infty} u(x) = 0
\end{equation}
(see \cite{Li}).
Set also $W_\ds = \text{Ext}^s(w_\ds)$, the $s$-harmonic extension of $w_\ds$.
Then we observe that extremal functions of Sobolev trace inequality
\begin{equation}\label{eq_sharp_trace}
\( \int_{\mr^N} |U(x,0)|^{2N \over N-2s} dx \)^{N-2s \over 2N}
\le \frac{\mathcal{S}_{N,s}}{\sqrt{\kappa_s}} \( \int_0^{\infty}\int_{\mr^N} t^{1-2s} |\nabla U(x,t)|^2 dx dt \)^{1 \over 2},
\end{equation}
have the form $U(x,t) = c W_\ds(x,t)$ for any $c > 0,\ \delta>0$ and $\sigma \in \mr^N$,
where $\kappa_s > 0$ is the constant defined in \eqref{eq_weighted_normal}.
Moreover, by its definition, $W_\ds$ solves
\begin{equation}\label{eq_limit}
\left\{ \begin{array}{ll} \text{div}\(t^{1-2s} \nabla U\)= 0 &\quad \text{in } \mr^{N+1}_+,\\
\pa_{\nu}^s U = U^p &\quad \text{on } \mr^N \times \{0\},\\
U = w_\ds &\quad \text{on } \mr^N \times \{0\}
\end{array} \right.
\end{equation}
and as an immediate consequence we have
\begin{equation}\label{eq-bubble-eq}
\kappa_s \int_{\mr^{N+1}_+} t^{1-2s}|\nabla W_\ds|^2 dxdt = \int_{\mr^N} w_\ds^{2N \over N-2s} dx.
\end{equation}

On the other hand, in the work of D\'avila, del Pino and Sire \cite{DDS},
it was revealed that 
the set of solutions bounded on $\Omega \times \{0\}$ to the equation
\begin{equation}\label{eq_lin_pro}
\left\{ \begin{array}{ll} \text{div}\(t^{1-2s} \nabla \Phi\)= 0 &\quad \text{in } \mr^{N+1}_+,\\
\pa_{\nu}^s \Phi = p w_\ds^{p-1} \Phi &\quad \text{on } \mr^N \times \{0\},
\end{array} \right.
\end{equation}
consists of the linear combinations of
\begin{equation}\label{eq_sol_lin}
Z^1_{\ds} := {\pa W_\ds \over \pa \sigma_1},\ \cdots,\ Z^N_{\ds} := {\pa W_\ds \over \pa \sigma_N} \quad \text{and} \quad Z^0_{\ds} := {\pa W_\ds \over \pa \delta}.
\end{equation}
This fact is crucial in applying the reduction method to our problem.
Hereafter, we will denote $w_{\delta} = w_{\delta,0}$, $W_{\delta} = W_{\delta, 0}$, $z^i_{\delta} = z^i_{\delta, 0}$ and $Z^i_{\delta} = Z^i_{\delta, 0}$ for $i = 0, \cdots, N$.

\subsection{Expansion of the metric near the boundary}
Suppose that $(X, \bg)$ is a compact Riemannian manifold and $0 \in M = \pa X$.
Let $x = (x_1, \cdots, x_N)$ be normal coordinates on $M$ at the point $0$ and $(x_1, \cdots, x_N, t)$ be the Fermi coordinates on $X$ at $0$ where $x_1, \cdots, x_N \in \mr$ and $t > 0$.
Also, we denote
\[\bg = dt^2 + h_{ij}(x,t)dx_idx_j\]
so that $h = \bg|_{TM}$. Then the following asymptotic expansion of the metric near 0 is valid.
\begin{lemma}\label{lemma_metric_exp} \cite[Lemma 3.1, 3.2]{Es}
For $x_1, \cdots, x_N$ and $t := x_{N+1}$ small, it holds that
\begin{align*}
\sqrt{|\bg|} &= \sqrt{|h|} = 1 - Ht + {1 \over 2} \(H^2-\|\pi\|_{h}^2-\text{Ric}(\pa_t)\)t^2 - H_ix_it
- {1 \over 6}R_{ij}x_ix_j + O\(|(x,t)|^3\)
\intertext{and}
h^{ij} &= \delta^{ij} + 2\pi^{ij}t-{1 \over 3}R_{\phantom{i}kl}^{i\phantom{kl}j}x_kx_l + h^{ij}_{\phantom{ij},(N+1)k}x_kt + \(3\pi^{ik}\pi^{\phantom{m}j}_m + R^{i\phantom{n}j}_{\phantom{i}n\phantom{j}n}\)t^2 + O\(|(x,t)|^3\)
\end{align*}
where $\pi$ is the second fundamental form of $M = \pa X$, $H$ is its trace, i.e., $N$ times of the mean curvature, $R_{ij}$ denotes a component of the Ricci tensor,
$R_{ijkl}$ is a component of the Riemannian tensor and Ric$(\pa_t) = g^{ij}R_{i(N+1)j(N+1)}$.
Also, the indices $i, j$ and $k$ run from 1 to $N$.
\end{lemma}

\section{Setting for the problem}\label{sec_scheme}
\subsection{The function spaces}
As before, let $(X^{N+1}, g^+)$ be an A.H. manifold with the boundary $(M^n, \hh)$ and $\rho$ the geodesic defining function, so that $(X, \bg)$ where $\bg = \rho^2g^+$ is a compact manifold.
Denote by $\hx$ the weighted Sobolev space endowed with the inner product
\[\la U, V \ra_{\hx} := \int_X \rho^{1-2s}\left[(\nabla U,\nabla V)_{\bg} +UV \right] dv_{\bg}\]
and the norm
\begin{equation}\label{eq_sobolev_norm}
\|U\|_{\hx} := \(\int_X \rho^{1-2s} \(|\nabla U|_{\bg}^2 +U^2\)dv_{\bg}\)^{1/2}.
\end{equation}
By applying \eqref{eq_sharp_trace} and the standard partition of unity argument, we obtain a manifold version of the weighted Sobolev trace inequality
\begin{equation}\label{eq_trace_mf}
\left\|U\right\|_{L^{\frac{2N}{N-2s}}(M)} \le C \left\|U\right\|_{\hx}
\end{equation}
where $C > 0$ is a constant determined by $s$, $N$ and $X$.
In addition, the embedding $\hx \hookrightarrow L^q(M)$ is compact for any $1 \le q < \frac{2N}{N-2s}$.
The next two lemmas provide equivalent norms to the $\hx$-norm.
\begin{lemma}\label{lem-equiv-2}
The norm $\( \int_X \rho^{1-2s}|\nabla U|_{\bar{g}}^2 dv_g + \int_M U^2 dv_{\hh}\)^{1/2}$  is equivalent to the norm $\| U \|_{\hx}$ defined in \eqref{eq_sobolev_norm}.
\end{lemma}
\begin{proof}
We first consider a function $U$ defined on $\overline{B_{2R}^+}$ for some $R > 0$ where $B_{2R}^+ = \{(x,t) \in \mr_+^{N+1} : |(x,t)| < 2R,\ t > 0\}$.
For each $0 \le t \le R$, using the elementary calculus and H\"older's inequality we have
\begin{align*}
|U(x,t)| &\le |U(x,0)| + \int_0^t |\pa_r U(x,r)| dr \le |U(x,0)| + \( \int_0^t r^{2s-1} dr \)^{1/2} \( \int_0^t r^{1-2s} |\pa_r U(x,r)|^2 dr \)^{1/2} \\
& = |U(x,0)| + \frac{t^{s}}{\sqrt{2s}} \( \int_0^R r^{1-2s} |\pa_r U(x,r)|^2 dr \)^{1/2}.
\end{align*}
For any given number $a \in (-1,1)$, we apply the above estimate to get
\begin{equation}\label{eq_bdry_est}
\begin{aligned}
&\int_{0}^R \int_{|x|\le R} t^a |U(x,t)|^2 dx\, dt \\
&\le 2 \(\int_0^R t^a dt\) \int_{|x| \le R} |U(x,0)|^2 dx + \frac{1}{s} \(\int_0^R t^{a+2s} dt\)\int_{|x|\le R} \int_0^R r^{1-2s} |\pa_r U(x,r)|^2 dr\, dx \\
&\le  C \( \int_{|x|\le R} |U(x,0)|^2 dx  +  \int_0^R\int_{|x|\le R} t^{1-2s} |\nabla U(x,t)|^2 dx\,dt\).
\end{aligned} \end{equation}
Employing this inequality with $a = 1-2s$ in each local chart, we can obtain that
\begin{equation*}
\( \int_X \rho^{1-2s}|U|^2 dv_{\bg} \)^{1/2} \le C\( \int_X \rho^{1-2s}|\nabla U|_{\bg}^2 dv_{\bg} + \int_M U^2 dv_{\hh}\)^{1/2}.
\end{equation*}
On the other hand, the weighted trace inequality \eqref{eq_trace_mf} and H\"older's inequality yield
\[\( \int_M |U|^2 dv_{\hh} \)^{1/2} \le C\( \int_X \rho^{1-2s}\(|\nabla U|_{\bar{g}}^2 + U^2\) dv_{\bg}\)^{1/2}.\]
These two estimates enable us to get the equivalence of the two norms, concluding the proof.
\end{proof}

\begin{lemma}\label{lem-equiv}
Suppose that the trace of the second fundamental form $H$ of $M = \pa X$ vanishes if $s \in [1/2, 1)$.
Under the assumption that \eqref{eq_coercive} holds,
\begin{equation}\label{eq_sobolev_norm_2}
\|U\|_{\tf} := \(\kappa_s \int_X \(\rho^{1-2s}|\nabla U|_{\bg}^2 + E(\rho)U^2\)dv_{\bg} + \int_M \tf U^2 dv_{\hh}\)^{1/2}
\end{equation}
gives an equivalent norm to \eqref{eq_sobolev_norm}.
Hence one can define the inner product $\la \cdot, \cdot \ra_{\tf}$ from the norm $\|\cdot\|_{\tf}$ through the polarization identity.
\end{lemma}
\begin{proof}
Suppose first that $s \in [1/2, 1)$.
In this case, the condition $H = 0$ is assumed, so $|E(\rho)| \le C\rho^{1-2s}$ by \eqref{eq_E_local_2}.
Using this fact and \eqref{eq_trace_mf} also, we immediately obtain that $\|U\|_{\hx} \ge C \|U\|_{\tf}$.
If $s \in (0, 1/2)$, then one can control the integral value of $U$ near the boundary by taking $a= -2s$ in \eqref{eq_bdry_est} and applying \eqref{eq_trace_mf}.
Additionally, by realizing that $\rho$ is bounded away from 0 in any compact subset of $X$, it is possible to manage the integral of $U$ in the interior of $X$.
Combining the both estimates, we deduce the same inequality $\|U\|_{\hx} \ge C \|U\|_{\tf}$.

Suppose that the opposite inequality does not hold.
Then there is a sequence $\{U_n\}_{n=1}^{\infty}$ such that $\|U_n\|_{\tf} \to 0$ as $n \to \infty$ but $\|U_n\|_{\hx} = 1$ for all $n \in \mathbb{N}$.
Let us first claim that $\int_X E(\rho)U_n^2 \to 0$.
By \eqref{eq_coercive}, we have $\int_X \rho^{1-2s} U_n^2 \to 0$, so the claim is verified at once if $H = 0$.
If $s \in (0, 1/2)$ and $H \ne 0$, then the main order of $E(\rho)$ is $\rho^{-2s}$ as \eqref{eq_E_local_2} indicates.
In this situation, we take $a < 1$ close to 1 and use the H\"older's inequality to get
\[\lim_{n \to \infty} \int_X \rho^{-2s} U_n^2 \le \lim_{n \to \infty} \(\int_X \rho^{1-2s} U_n^2\)^{\eta} \(\int_X \rho^{-a} U_n^2\)^{1-\eta} = 0\]
where $\eta = {a-2s \over a-2s+1} \in (0,1)$, so we can justify our claim again.
Observe that $\|U_n\|_{\hx} = 1$, \eqref{eq_trace_mf} and \eqref{eq_bdry_est} guarantee boundedness of the value $\left\{\int_X \rho^{-a} U_n^2\right\}_{n=1}^{\infty}$.
Now if we let $U_{\infty}$ be the $\hx$-weak limit of $U_n$, then $U_{\infty} \equiv 0$.
Thus compactness of the trace embedding gives us that $\int_M \tf U_n^2 \to \int_M \tf U_{\infty}^2 = 0$.
However, it is a contradiction because previous computations show that $\int_X \rho^{1-2s} |\nabla U_n|^2$ should converge to both 0 and 1.
This proves that $\|U\|_{\tf} \ge C \|U\|_{\hx}$.
\end{proof}

\noindent By \eqref{eq_trace_mf}, we know that the trace operator
$i: \hx \to L^{p+1}(M)$ given as $i(U) = U|_M := u$ is well-defined and continuous.
Thus the adjoint operator $i_{\tf}^*: L^{p+1 \over p}(M) \to \hx$ defined by the equation
\begin{equation}\label{eq_adjoint}
\left\{\begin{array}{ll}
-\textnormal{div}\(\rho^{1-2s} \nabla U\) + E(\rho) U = 0 &\text{in } (X, \bg),\\
\pa_\nu^s U = v - \tf u &\text{on } (M, \hh),\\
U = u &\text{on } (M, \hh),
\end{array}\right.
\end{equation}
with $U = i_{\tf}^*(v)$ is bounded in light of Lemma \ref{lem-equiv}.
Furthermore, $i: \hx \to L^q(M) \supset L^{p+1}(M)$ for $1 \le q < p+1$ is compact.

\medskip
On the other hand, in order to take account into the supercritical problem $(1_+)$ or $(15_+)$,
we must restrict the space $\hx$ so that the trace of the each element belongs to $L^{p+1+\ep}(M)$ for $\ep > 0$ small. Set
\begin{equation}\label{eq_q_ep}
q_{\ep} = (p+1) + {N \over 2s} \ep, \quad \text{which implies} \quad {q_{\ep} \over p+\ep} = {Nq_{\ep} \over N+2sq_{\ep}}.
\end{equation}
Then let us introduce a Banach space
\begin{equation}\label{eq_H_ep}
\mh_{\ep} = \left\{U \in \hx: i(U) \in L^{q_{\ep}}(M)\right\}
\end{equation}
equipped with the norm $\|\cdot\|_{\tf, \ep}$ defined by
\begin{equation}\label{eq_H_eq_norm}
\|U\|_{\tf, \ep} = \|U\|_{\tf} + \|i(U)\|_{L^{q_{\ep}}(M)} \quad \text{for any } U \in \mh_{\ep}.
\end{equation}
The following estimate explains why it is plausible to work with the space $\mh_{\ep}$.
\begin{lemma}\label{lemma_trace}
Suppose that $N > 2s$ and $v \in L^{q_1}(M)$ for some $q_1 \in (1, {N \over 2s})$.
If $U = i_{\tf}^*(v)$ and $u = i(U)$, then there exists $C = C(q_1) > 0$ such that
\[\|u\|_{L^{q_2}(M)} \le C \left\|v\right\|_{L^{q_1}(M)}\]
with $q_2 > {N \over N-2s}$ satisfying $\frac{1}{q_2} = \frac{1}{q_1} - \frac{2s}{N}$.
In other words, we have
\[\|u\|_{L^q(M)} \le C \|v\|_{L^{\frac{Nq}{N+2sq}}(M)}\]
for any $q \in ({N \over N-2s}, \infty)$.
\end{lemma}
\begin{proof}
Instead of giving consideration to \eqref{eq_adjoint} directly, we shall use the observation coming from Propositions \ref{prop_ext} and \ref{prop_ext_2}
that $\wu = (\rho^*/\rho)^{\verb"z"-N}U$ is a solution of \eqref{eq_ext_1} and $\wu = U = u$ on $M$.
For any number $L > 0$, let us denote $\wu_L = \min\big\{|\wu|, L\big\}$.
Due to \eqref{eq_ext_2}, if we multiply \eqref{eq_ext_1} by $\wu_L^{\beta-1} \wu$ for some $\beta > 1$, we get
\[\kappa_s \int_X (\rho^*)^{1-2s} \(\nabla \wu, \nabla \(\wu_L^{\beta-1} \wu\)\)_{g^*} dv_{g^*}
= \int_M v u_L^{\beta-1}u dv_{\hh} - \int_M \(\tf + Q^s_{\hh}\) u_L^{\beta-1}u^2 dv_{\hh}\]
where $u_L = \min\{|u|, L\}$. Therefore we have
\begin{align*}
\int_{X} (\rho^*)^{1-2s} \left|\nabla \(\wu_L^{\frac{\beta-1}{2}} \wu\)\right|_{g^*}^2 dv_{g^*}
&\le C \left\|u_L^{\beta-1}u \right\|_{L^{\frac{(\beta+1)(p+1)}{2\beta}}(M)} \|v\|_{L^{q'} (M)} + C \|u\|_{L^{\beta+1}(M)}^{\beta+1}
\\
&\le  C \left\| \(u_L^{\frac{\beta-1}{2}}u\)^{\frac{2\beta}{\beta+1}}\right\|_{L^{\frac{(\beta+1)(p+1)}{2\beta}}(M)} \|v\|_{L^{q'} (M)}
+ C \|u\|_{L^{\beta+1}(M)}^{\beta+1}
\\
&\le {1 \over C} \left\|u_L^{\frac{\beta-1}{2}}u\right\|_{L^{p+1}(M)}^{2} + C \|v\|_{L^{q'} (M)}^{\beta+1} + C \|u\|_{L^{\beta+1}(M)}^{\beta+1},
\end{align*}
where $q'$ satisfies $\frac{1}{q'} + \frac{(N-2s)\beta}{N(\beta+1)} =1$ and $C>0$ is a large number determined by $N$ and $s$.
Also, we used Young's inequality to derive the third inequality.
Using this, Lemma \ref{lem-equiv-2} and the weighted trace inequality, we get
\begin{equation}\label{eq-trace-0}
\begin{aligned}
\left\|u_L^{\frac{\beta-1}{2}}u\right\|_{L^{p+1}(M)}^{2} &\le \int_{X} (\rho^*)^{1-2s} \left|\nabla \(\wu_L^{\frac{\beta-1}{2}}\wu\)\right|_{g^*}^2 dv_{g^*} + \int_{M} \(u_L^{\frac{\beta-1}{2}}u\)^2 dv_{\hh}
\\
& \le \frac{1}{C}  \left\|u_L^{\frac{\beta-1}{2}}u\right\|_{L^{p+1}(M)}^{2} + C \|v\|_{L^{q'}(M)}^{\beta+1} + C \|u\|_{L^{\beta+1}(M)}^{\beta+1}.
\end{aligned}
\end{equation}
Taking $L\to \infty$ in this estimate, we may deduce
\[\|u\|_{L^{\frac{N(\beta+1)}{N-2s}}(M)} \le C\(\|v\|_{L^{q'}(M)} + \|u\|_{L^{\beta+1}(M)}\).\]
Letting $q = \frac{N(\beta+1)}{N-2s}$ we have
\begin{equation}\label{eq-trace-1}
\|u\|_{L^q (M)} \le C\(\|v\|_{L^{q'}(M)} + \|u\|_{L^{\frac{(N-2s)q}{N}}(M)}\).
\end{equation}
One may check that $\frac{1}{q} = \frac{1}{q'} - \frac{2s}{N}$.
Besides, since we took $\beta > 1$, it holds that $q' > {2N \over N+2s}$ and $q > p+1$.
On the other hand, if we test \eqref{eq_ext_1} with $(\wu^L)^{\beta-1} \wu$ for $0 < \beta \le 1$ where $\wu^L := \max\big\{|\wu|, L\big\}$
and follow the above argument except taking $L \to 0$ in \eqref{eq-trace-0} instead $L \to \infty$,
then we obtain \eqref{eq-trace-1} for $1 < q' \le {2N \over N+2s}$ and ${N \over N-2s} < q \le p+1$.

We claim further that $\|u\|_{L^q(M)} \le C_1 \|v\|_{L^{q'}(M)}$ holds for some $C_1 >0$.
To show this inequality, we assume that it does not hold for any $C_1$.
Then, we can find a sequence of functions $v_n \in L^{q'}(M)$, $U_n = i_{\bar{f}}^{*} (v_n)$ and $u_n = i(U_n)$
such that $\|u_n\|_{L^q (M)} =1$ and $\lim_{n \to \infty}\| v_n\|_{L^{q'}(M)} =0$.
By the compactness property whose proof is postponed to below, $u_n$ converges strongly in $L^{\frac{(N-2s)q}{N}}(M)$.
We let $u_0$ be its limit.
Applying \eqref{eq-trace-1} with $u_n$ and $v_n$, and then taking the limit $n \to \infty$, we obtain
\begin{equation}\label{eq-trace-2}
1 \le C\(\lim_{n \to \infty} \|v_n\|_{L^{q'} (M)} + \| u_n\|_{L^{\frac{(N-2s)q}{N}}(M)}\) = C\| u_0\|_{L^{\frac{(N-2s)q}{N}}(M)}.
\end{equation}
On the other hand, by employing Lemma \ref{lem-equiv}, the weighted trace inequality and H\"older's inequality, we find
\[\|u_n\|_{L^{\frac{2N}{N-2s}}(M)} \le  C \| U_n \|_{\tilde{f}} \le C \| v_n \|_{L^{\frac{2N}{N+2s}}(M)} \le C\|v_n\|_{L^{q'} (M)}.\]
From this estimate and $\lim_{n \to \infty} \|v_n\|_{L^{q'}(M)} = 0$, we have $\|u_0\|_{L^{\frac{2N}{N-2s}}(M)} = \lim_{n \to \infty}\|u_n\|_{L^{\frac{2N}{N-2s}}(M)} = 0$,
implying $u_0 \equiv 0$.
However it contradicts to \eqref{eq-trace-2}.
Hence the assertion that $\|u\|_{L^q (M)} \le C_1 \|v\|_{L^{q'}(M)}$ should hold for some $C_1 >0$.

We are left to prove the compactness of $\{u_n\}_{n=1}^{\infty}$ in $L^{\frac{(N-2s)q}{N}}(M)$.
By \eqref{eq-trace-0}, we get
\[\int_X \rho^{1-2s}\left|\nabla |U_n|^{\frac{\beta+1}{2}}\right|_{\bg}^2 dv_{\bg} + \int_M |U_n|^{\beta+1} dv_{\hh}
\le C\( \|v_n\|_{L^{q'} (M)} + \| u_n\|_{L^{\beta+1}(M)}\)^{\beta+1}.\]
Owing to Lemma \ref{lem-equiv-2}, it follows that $\left\{|U_n|^{\frac{\beta+1}{2}}\right\}_{n=1}^{\infty}$ is a bounded subset of $\hx$.
Thus $\left\{|U_n|^{\frac{\beta+1}{2}}\right\}_{n=1}^{\infty}$ is a compact set in $L^{\frac{2N}{N-2s}-\zeta}(M)$ for any small $\zeta > 0$, which in turn implies that
$\{U_n\}_{n=1}^{\infty}$ is a compact set in $L^{\frac{N(\beta+1)}{N-2s}-\zeta}(M)= L^{q-\zeta}(M)$ for every small $\zeta > 0$, hence in $L^{\frac{(N-2s)q}{N}}(M)$.
The proof is finished.
\end{proof}

\begin{cor}\label{cor_adj}
Fix any $q > {2N \over N+2s}$.
Then the adjoint map $i_f^*: L^q(M) \to \mh_{\ep}$ is compact for sufficiently small $\ep > 0$.
\end{cor}
\begin{proof}
It easily follows from the previous lemma and its proof.
We leave the details to the reader.
\end{proof}

\noindent By Lemma \ref{lemma_trace}, if $u \in L^{q_{\ep} \over p+\ep}(M)$, then $i(i_{\tf}^*(u)) \in L^{q_{\ep}}(M)$.
Hence one may attempt to solve equation $(1_+)$ by writing
\[U = i^*_{\tf}\(u^{p+\ep}\) \quad \text{and} \quad U = u > 0 \quad \text{on } M\]
for $U \in \mh_{\ep}$.

\medskip
To unify the notation, we will use $(\mh_{\ep}, \|\cdot\|_{\tf, \ep})$ to denote $(\hx, \|\cdot\|_{\tf})$ from now even if we study the subcritical problem $(1_-)$ and the critical one \eqref{eq_main_2}.
Notice that if equations $(1_-)$ and \eqref{eq_main_2} are considered,
then $q_{\ep}$ in \eqref{eq_q_ep} should be read as ${2N \over N-2s} - {N \over 2s}\ep$ and ${2N \over N-2s}$, respectively.
Hence in this case the Banach spaces $(\mh_{\ep}, \|\cdot\|_{\tf, \ep})$ (defined according to \eqref{eq_H_ep} and \eqref{eq_H_eq_norm}) and $(\hx, \|\cdot\|_{\tf})$ are equivalent to each other, justifying our expression.

\subsection{The approximate solutions}
Recalling the number $r_1$ selected in \eqref{eq_E_local}, we choose $r_0 < r_1$ a positive number less than the quarter of the injectivity radius of $(M,\hh)$.
Let $\chi_1: (0,\infty) \to [0,1]$ be a smooth function such that $\chi_1 = 1$ in $(0, r_0)$ and 0 in $(2r_0, \infty)$.
Noting that any element $z \in X$ near the boundary can be denoted as $z = (\hs, \rho)$ for some $\hs \in M$ and $\rho \in (0, \infty)$, we define the function $\mw_\ds$ on $X$ (provided $\delta > 0$ and $\sigma \in M$) by
\begin{equation}\label{eq_first_approx}
\mw_\ds(z) = \mw_\ds(\hs,\rho) = \begin{cases} \chi_1(d(z,\sigma))W_{\delta}\(\exp_{\sigma}^{-1}(\hs), \rho\) &\text{if } d(z,\sigma) < 2r_0 \text{ for some } \sigma \in M,\\
0 &\text{otherwise},
\end{cases} \end{equation}
where $W_{\delta} = \text{Ext}^s(w_{\delta})$ is the function defined in Subsection \ref{subsec_trace_ineq}, $d_M(\cdot, \sigma)$ denotes the geodesic distance from $\sigma$ on $(M, \hh)$,
$d(\cdot, \sigma)$ is a positive function defined near the boundary of $(X, \bg)$ by the relation $d(z, \sigma)^2 = d((\hs,\rho),\sigma)^2 = d_M(\hs,\sigma)^2+\rho^2$ and exp is the exponential map on $(M, \hh)$.
Thus the parameter $\delta$ can be regarded as a concentration rate, while $\sigma$ expresses a blow-up point.
We set $\delta = \ep^{\alpha} \lambda$ where $\lambda > 0$ is an $\ep$-independent number.
The number $\alpha$ is chosen to be
\begin{equation}\label{eq_alpha}
\alpha = \begin{cases}
1/(2s) &\text{for problems \eqref{eq_main_12}},\\
1/(1-2s) &\text{for problem \eqref{eq_main_21}}.
\end{cases}
\end{equation}

In this paper, we search for solutions of \eqref{eq_main_12} and \eqref{eq_main_21} of the form $\mw_\els + \Phi$
where $\Phi$ is a function defined on $X$ whose $\mh_{\ep}$-norm is sufficiently small.
Because we regard the equations as perturbations of the \textit{limit equation} \eqref{eq_limit}, it is important to understand their linearized equations.
Hence it is natural to introduce
\[\mz^i_\ds(z) = \mz^i_\ds(\hs,\rho) = \begin{cases} \chi_1(d(z,\sigma))Z^i_{\delta}\(\exp_{\sigma}^{-1}(\hs), \rho\) &\text{if } d(z,\sigma) < 2r_0 \text{ for some } \sigma \in M,\\
0 &\text{otherwise},
\end{cases}\]
for $i = 0, \cdots, N$, where $Z^i_\ds$ is the function whose definition is presented in \eqref{eq_sol_lin}.
For each $\ep > 0$, let us also define the subspace of $\mh_{\ep}$
\[K_\ls^{\ep} = \text{Span}\left\{\mz^i_\els : i = 0, \cdots, N\right\}\]
and its orthogonal complement with respect to the inner product $\la \cdot, \cdot \ra_{\tf}$
\[\(K_\ls^{\ep}\)^\perp = \left\{U \in \mh_{\ep} : \la U, \mz^i_\els \ra_{\tf} = 0:  i = 0, \cdots, N \right\}.\]
Furthermore, denote by
\[\Pi_\ls^{\ep}: \mh_{\ep} \to K_\ls^{\ep} \quad \text{and} \quad \(\Pi_\ls^{\ep}\)^{\perp}: \mh_{\ep} \to \(K_\ls^{\ep}\)^\perp\]
the orthogonal projections onto $K_\ls^{\ep}$ and $(K_\ls^{\ep})^\perp$, respectively.

As mentioned before, we will apply the finite dimensional reduction method.
Namely, for a small fixed $\ep > 0$, we first solve an intermediate problem (in Section \ref{sec_inter})
\begin{equation}\label{eq_inter_1}
\(\Pi_\ls^{\ep}\)^{\perp}\left[(\mw_\els + \Phi_\els) -i_{\tf}^*\(i\(g_{\ep}(\mw_\els + \Phi_\els)\)\)\right] = 0 \end{equation}
for each parameter $(\ls) \in (0,\infty) \times M$ by employing the contraction mapping theorem,
where
\begin{equation}\label{eq_gep}
\left\{ \begin{array}{lll}
g_{\ep}(u) = u_+^{p \pm \ep} &\text{and } \tf = f &\text{if we consider } \eqref{eq_main_12},\\
g_{\ep}(u) = u_+^p - \ep fu &\text{and } \tf = 0 &\text{if we consider } \eqref{eq_main_21}.
\end{array}\right..
\end{equation}
Then we choose an appropriate $(\lambda_{\ep}, \sigma_{\ep})$ which makes
\begin{equation}\label{eq_inter_2}
\Pi_{\lambda_{\ep}, \sigma_{\ep}}^{\ep} \left[(\mw_\elep + \Phi_\elep) - i_{\tf}^*\(i\(g_{\ep}(\mw_\elep + \Phi_\elep)\)\)\right] = 0
\end{equation}
by finding a critical point of a suitable (localized) energy functional on $(0,\infty) \times M$ corresponding to the above problem \eqref{eq_inter_2}.
This is conducted in Section \ref{sec_energy}.
Observe that we modified the nonlinear term in \eqref{eq_inter_1} and \eqref{eq_inter_2} because we want to find a positive solution.

\medskip
Before concluding this section, we provide a lemma regarding the decay property of $W_{\delta}$ and $Z_{\delta}^i$, which will be used throughout the paper.
We defer its proof to Appendix \ref{sec_appen}.
\begin{lemma}\label{lemma_decay}
Assume that $N > 2s$, fix any $0 < R_1 < R_2$ and set $A^+_{(R_1,R_2)} = B^+_{R_2} \setminus B^+_{R_1}$. Then as $\delta \to 0$ we have the following estimates.
\begin{equation}\label{eq_decay_1}
\begin{aligned}
\int_{\mr^{N+1}_+ \setminus B^+_{R_1}} t^{1-2s} |\nabla W_{\delta}|^2 dxdt & = O\(\delta^{N-2s}\).
\\
\int_{B^+_{R_1}} t^{2-2s} |\nabla W_{\delta}|^2 dxdt & = \begin{cases}
O\(\delta\) &\text{for } N > 2s+1, \\
O\(\delta |\log \delta|\) &\text{for } N = 2s+1, \\
O\(\delta^{N-2s}\) &\text{for } N < 2s+1.
\end{cases}
\\
\int_{A^+_{(R_1,R_2)}} t^{1-2s} W_{\delta}^2 dxdt & = \begin{cases}
O\(\delta^{N-2s}\) &\text{for } N \ne 2s+2, \\
O\(\delta^2 |\log \delta|\) &\text{for } N = 2s+2.
\end{cases}
\end{aligned}
\end{equation}
Besides, the followings are also true.
\begin{equation}\label{eq_decay_3}
\begin{aligned}
\int_{\mr^{N+1}_+ \setminus B^+_{R_1}} t^{1-2s} \left|\nabla Z_{\delta}^i\right|^2 dxdt &= \begin{cases}
O\(\delta^{N-2s}\) &\text{for } i = 1, \cdots, N, \\
O\(\delta^{N-2s-2}\) &\text{for } i = 0.
\end{cases}
\\
\int_{A^+_{(R_1,R_2)}} t^{1-2s} \(Z_{\delta}^i\)^2 dxdt &= \begin{cases}
O\(\delta^{N-2s}\) &\text{for } i = 1, \cdots, N,\\
O\(\delta^{N-2s-2}\) &\text{for } i = 0 \text{ and } N \ne 2s+2,\\
O\(|\log \delta|\) &\text{for } i = 0 \text{ and } N = 2s+2.
\end{cases}
\end{aligned}
\end{equation}
We also know
\begin{equation}\label{eq_decay_2}
\int_{B^+_{R_1}} t^{1-2s} O\(|(x,t)|^2\) |\nabla W_{\delta}|^2 dx dt = \begin{cases}
O\(\delta^2\) &\text{for } N > 2s+2, \\
O\(\delta^2 |\log \delta|\) &\text{for } N = 2s+2, \\
O\(\delta^{N-2s}\) &\text{for } N < 2s+2.
\end{cases} \end{equation}
\end{lemma}

\section{Solvability of the intermediate problem}\label{sec_inter}
This section is devoted to solvability of the intermediate problem \eqref{eq_inter_1}.
\subsection{Estimates for the error}\label{subsec_error}
In this subsection, we shall obtain a uniform bound of the $\mh_{\ep}$-norm of the error term $\mw_\els-i_{\tf}^*(i(g_{\ep}(\mw_\els)))$
where $(\ls) \in (\lambda_1^{-1}, \lambda_1) \times M$ and $\ep > 0$ small, given any fixed number $\lambda_1 > 0$.
The positive number $\alpha$ was set in \eqref{eq_alpha}.

\begin{lemma}\label{lemma_error_est}
Assume that $N > \max\{4s,1\}$ for \eqref{eq_main_12} and $N \ge 2$ 
for \eqref{eq_main_21}. Given a fixed $\lambda_1 > 0$, it holds that
\begin{equation}\label{eq_error_est}
\left\|\mw_\els - i_{\tf}^*\(i(g_{\ep}(\mw_\els))\)\right\|_{\tf, \ep} = O\(\ep^{\gamma}\)
\end{equation}
where
\begin{equation}\label{eq_error_est_0}
\gamma = \begin{cases}
1-\zeta_0 &\quad \text{for problems \eqref{eq_main_12} if } 0 < s < {1 \over 3}, \\
\frac{1-s}{2s}-\zeta_0 &\quad \text{for problems \eqref{eq_main_12} if } {1 \over 3} \le s < {1 \over 2},\\
{N-2s \over 4s} - \zeta_0 &\quad \text{for problems \eqref{eq_main_12} with } 4s < N \le 2s+2 \text{ if } {1 \over 2} \le s < 1, \\
{1 \over 2s} - \zeta_0 &\quad \text{for problems \eqref{eq_main_12} with } N > 2s + 2 \text{ if } {1 \over 2} \le s < 1,\\
\frac{1-s}{1-2s}-\zeta_0 &\quad \text{for problem \eqref{eq_main_21}}
\end{cases} \end{equation}
uniformly $(\ls) \in (\lambda_1^{-1}, \lambda_1) \times M$.
Here $\zeta_0 > 0$ can be taken to be arbitrarily small.
\end{lemma}
\noindent Before starting the proof, we remark that $\gamma > 1/2$ for problems \eqref{eq_main_12}, while $\gamma > 1/(2(1-2s))$ for problem \eqref{eq_main_21}.
\begin{proof}
Let us take into account the subcritical problem $(15_-)$,
recalling the notation $\delta = \ep^{1 \over 2s} \lambda \in \(\ep^{1 \over 2s}\lambda_1^{-1}, \ep^{1 \over 2s}\lambda_1\)$.
Here we will use the dual characterization of the norm
\[\|U\|_f = \sup\left\{\la U, \Phi \ra_f: \|\Phi\|_f \le 1\right\}\]
which holds for any $U \in \hx$.
For a fixed $\Phi \in \hx$ such that $\|\Phi\|_f \le 1$, we have
\begin{multline}\label{eq_error_est_1}
\la \mw_\ds, \Phi \ra_f - \la i(\mw_\ds^{p \pm \ep}), \phi \ra_{L^{p \over p+1}(M)}\\
= \kappa_s \int_{B_{\bg}^+(\sigma, 2r_0)} \(\rho^{1-2s} (\nabla \mw_\ds, \nabla \Phi)_{\bg} + E(\rho) \mw_\ds \Phi\) dv_{\bg} + \int_{B_{\hh}(\sigma, 2r_0)} \(f \mw_\ds - \mw_\ds^{p \pm \ep}\) \phi dv_{\hh}
\end{multline}
where
\begin{equation}\label{eq_B_bg}
B_{\bg}^+(\sigma, 2r_0) := \{z \in X: d(z,\sigma) < 2r_0\}, \quad B_{\hh}(\sigma, 2r_0) := \{\hs \in M: d_M(\hs,\sigma) < 2r_0\}
\end{equation}
and $\phi = i(\Phi)$.
Note that in Section \ref{sec_scheme} the distance functions $d(\cdot, \sigma)$ and $d_M(\cdot, \sigma)$ were introduced in setting the first approximation $\mw_\ds$ for a solution,
for each fixed $\sigma \in M$ (see \eqref{eq_first_approx}).
Since the domains of the above integrations are small neighborhoods of the point $\sigma$ in $X$ and $M$, respectively,
we may replace $\Phi$ by $\chi_1(d(\cdot,\sigma)/2)\Phi$ for instance without affecting on the value of the integrations,
where $\chi_1$ is a cut-off function introduced for \eqref{eq_first_approx}.
Moreover, by the equivalence of two norms $\|\cdot\|_f$ and $\|\cdot\|_{\hx}$, it can be easily seen that $\|\chi_1(d(\cdot,\sigma)/2)\Phi\|_f \le  C_0\|\Phi\|_f \le C_0$ where $C_0 > 0$ is a number not relying on the choice of $\Phi$.
Therefore, to obtain \eqref{eq_error_est}, we may without any loss of generality regard $\Phi$ (or $\phi$) as a function on $\mr^{N+1}_+$ (or $\mr^N$)
and assume that its support is contained in $\overline{B_{\bg}^+} := \overline{B_{\bg}^+(\sigma, 4r_0)} \subset \overline{\mr^{N+1}_+}$ $\(\text{or } \overline{B_{\hh}} := \overline{B_{\hh}(\sigma, 4r_0)} \subset \mr^N\)$.

Now we shall estimate each of the right-hand side of \eqref{eq_error_est_1}.
For this objective, we denote $\Phi_{\delta^{-1}}(z) = \delta^{N-2s \over 2} \Phi(\delta z)$ for all $z \in \mr^{N+1}_+$ and $\phi_{\delta^{-1}} = i(\Phi_{\delta^{-1}})$, for which it holds that
\begin{equation}\label{eq_scale_inv}
\left\|\Phi_{\delta^{-1}}\right\|_{\dmr}^2 = \int_{\mr^{N+1}_+} t^{1-2s} |\nabla \Phi_{\delta^{-1}}(z)|^2 dxdt \le C
\end{equation}
by the scaling invariance.
Firstly, from \eqref{eq_decay_1} and the estimate that
\begin{align*}
\int_{B_{\bg}^+} \rho^{1-2s} |z||\nabla \mw_\ds||\nabla \Phi| dz
&\le C \left[ \( \int_{B^+_{2r_0} \setminus B^+_{r_0}} t^{1-2s} W_{\delta}^2 dz\)^{1 \over 2}
+ \( \int_{B^+_{2r_0}} t^{1-2s} |z|^2|\nabla W_{\delta}|^2 dz\)^{1 \over 2} \right]
\\
&= \begin{cases}
O(\delta) = O\(\ep^{1 \over 2s}\) &\text{if } N > 2s + 2,\\
O\(\delta |\log \delta|^{1 \over 2}\) = O\(\ep^{1 \over 2s}|\log \ep|^{1 \over 2}\) &\text{if } N = 2s + 2,\\
O\(\delta^{N-2s \over 2}\) = O\(\ep^{N-2s \over 4s}\) &\text{if } N < 2s + 2,
\end{cases} \end{align*}
we find
\begin{equation}\label{eq_error_est_2}
\begin{aligned}
&\kappa_s \int_{B_{\bg}^+} \rho^{1-2s} (\nabla \mw_\ds, \nabla \Phi)_{\bg} dv_{\bg}\\
&= \kappa_s \int_{\mr^{N+1}_+} t^{1-2s} \nabla W_{\delta} \cdot \nabla \Phi dz +
O\(\int_{B_{\bg}^+} \rho^{1-2s} |z||\nabla \mw_\ds||\nabla \Phi| dz\) +
\begin{cases}
O\(\delta^{N-2s \over 2}\) &\text{if } N \ne 2s + 2,\\
O\(\delta |\log \delta|^{1 \over 2}\) &\text{if } N = 2s + 2,
\end{cases}
\\
&= \kappa_s \int_{\mr^{N+1}_+} t^{1-2s} \nabla W_1 \cdot \nabla \Phi_{\delta^{-1}} dz + \begin{cases}
O(\delta) &\text{if } N > 2s + 2,\\
O\(\delta |\log \delta|^{1 \over 2}\) &\text{if } N = 2s + 2,\\
O\(\delta^{N-2s \over 2}\) &\text{if } N < 2s + 2,
\end{cases} \\
&= \int_{\mr^N} w_1^p(x) \phi_{\delta^{-1}}(x) dx + \begin{cases}
O\(\ep^{1 \over 2s}\) &\text{if } N > 2s + 2,\\
O\(\ep^{1 \over 2s}|\log \ep|^{1 \over 2}\) &\text{if } N = 2s + 2,\\
O\(\ep^{N-2s \over 4s}\) &\text{if } N < 2s + 2.
\end{cases}
\end{aligned}
\end{equation}
Also, if $1/2 \le s < 1$ and $H = 0$, then \eqref{eq_E_local_2} and the Cauchy-Schwarz inequality imply
\begin{equation}\label{eq_error_est_3}
\begin{aligned}
& \left| \kappa_s \int_{B_{\bg}^+} E(\rho) \mw_\ds \Phi dv_{\bg} \right|
\le C \(\int_{B_{\bg}^+} \rho^{1-2s} \mw_\ds^2 dv_{\bg}\)^{1 \over 2} \cdot \(\int_{B_{\bg}^+} \rho^{1-2s} \Phi^2 dv_{\bg}\)^{1 \over 2}\\
& \le C \(\int_{B^+_{2r_0}} t^{1-2s} W_{\delta}^2(z)dz\)^{1 \over 2}
= \begin{cases}
O(\delta) = O\(\ep^{1 \over 2s}\) &\text{if } N > 2s + 2,\\
O\(\delta |\log \delta|^{1 \over 2}\) = O\(\ep^{1 \over 2s}|\log \ep|^{1 \over 2}\) &\text{if } N = 2s + 2,\\
O\(\delta^{N-2s \over 2}\) = O\(\ep^{N-2s \over 4s}\) &\text{if } N < 2s + 2.
\end{cases} \end{aligned} \end{equation}
In the case that $0 < s < 1/2$, we take $\zeta_1 > 0$ small enough so that $1-(s+\zeta_1) > 1/2$. It follows that
\begin{equation}\label{eq_error_est_32}
\begin{aligned}
& \left| \kappa_s \int_{B_{\bg}^+} E(\rho) \mw_\ds \Phi dv_{\bg} \right| \le C \int_{B_{\bg}^+}\rho^{-2s} |\mw_\ds ||\Phi | dv_{\bg}
\\
&\le C \(\int_{B_{\bg}^+} \rho^{1-2s-2(s+\zeta_1)} \mw_\ds^2 dv_{\bg}\)^{1 \over 2} \cdot \(\int_{B_{\bg}^+} \rho^{-1+2\zeta_1} \Phi^2 dv_{\bg}\)^{1 \over 2}\\
& \le C \(\int_{B^+_{2r_0}} t^{1-2s-2(s+\zeta_1)} W_{\delta}^2(z)dz\)^{1 \over 2}
= O\(\delta^{1-(s+\zeta_1)}\) = O\(\ep^{(1-(s+\zeta_1))/2s}\) \quad \text{for } N \ge 2.
\end{aligned} \end{equation}
On the other hand, if $\zeta_2$ is a number chosen to be
\[\zeta_2 = \begin{cases}
{N \over N-2s} + \zeta_2' &\text{for } 4s < N \le 6s \text{ where } \zeta_2' > 0 \text{ is arbitrarily small},\\
{2N \over N+2s} &\text{for } N > 6s,
\end{cases}\]
then thanks to the Sobolev trace inequality \eqref{eq_trace_mf}, it can be computed that
\begin{equation}\label{eq_error_est_4}
\begin{aligned}
\left|\int_{B_{\hh}} f \mw_\ds \phi dv_{\hh}\right| &\le C\|f\|_{L^{\infty}(M)}\|w_{\delta}\|_{L^{\zeta_2}(\mr^N)}\|\Phi\|_{\hx} \\
&= \begin{cases}
O\(\delta^{{N-2s \over 2} - \zeta_2''}\) = O\(\ep^{{N-2s \over 4s} - {\zeta_2'' \over 2s}}\) &\text{for } 4s < N \le 6s,\\
O\(\delta^{2s}\) = O(\ep) &\text{for } N > 6s.
\end{cases} \end{aligned} \end{equation}
Here $\zeta_2'' > 0$ is again a small number depending on the selection of $\zeta_2'$. Moreover one has
\begin{equation}\label{eq_error_est_5}
- \int_{B_{\hh}} \mw_\ds^{p \pm \ep}\phi dv_{\hh} = - \int_{B_{\hh}} \mw_\ds^p \phi dv_{\hh} + O(\ep |\log \ep|)
= - \int_{\mr^N} w_1^p(x) \phi_{\delta^{-1}}(x) dx + O(\ep |\log \ep|).
\end{equation}
Consequently, combining all computations \eqref{eq_error_est_1} and \eqref{eq_error_est_2}-\eqref{eq_error_est_5}, we obtain the validity of the first estimate of \eqref{eq_error_est}.

The error estimate \eqref{eq_error_est} for problem \eqref{eq_main_21} can be handled in a similar way and we omit it.

Now we are left to handle the supercritical problems $(15_+)$.
To obtain the conclusion, it suffices to show that
\begin{equation}\label{eq_error_0}
\left\|w_\ds - i\(i_{\tf}^*(g_{\ep}(w_\ds))\)\right\|_{L^{q_{\ep}}(M)} = O\(\ep^{\gamma}\).
\end{equation}
By the trace inequality \eqref{eq_trace_mf} and the computations made above, we have
\begin{equation}\label{eq_error_1}
\begin{aligned}
&\left\|w_\ds - i\(i_{\tf}^*(g_{\ep}(w_\ds))\)\right\|_{L^{q_{\ep}}(M)}
\\
&\le C\left\|w_\ds - i\(i_{\tf}^*(g_{\ep}(w_\ds))\)\right\|_{L^{p+1}(M)}^{1-r_{\ep}}
\cdot \left\|w_\ds - i\(i_{\tf}^*(g_{\ep}(w_\ds))\)\right\|_{L^{2(p+1)}(M)}^{r_{\ep}}
\\
&\le  C\left\|\mw_\ds - i_{\tf}^*\(g_{\ep}(w_\ds)\)\right\|_f^{1-r_{\ep}}
\cdot \left\|w_\ds - i\(i_{\tf}^*(g_{\ep}(w_\ds))\)\right\|_{L^{2(p+1)}(M)}^{r_{\ep}}
\\
&\le  C \ep^{\gamma(1-r_{\ep})} \cdot \left\|w_\ds - i\(i_{\tf}^*(g_{\ep}(w_\ds))\)\right\|_{L^{2(p+1)}(M)}^{r_{\ep}}
\end{aligned}
\end{equation}
where $r_{\ep} \in (0,1)$ satisfies
\[\frac{1-r_{\ep}}{p+1} + \frac{r_{\ep}}{2(p+1)} = \frac{1}{q_{\ep}},\]
which leads to $r_{\ep}= \frac{N}{s\left[(p+1) + \frac{N}{2s}\ep\right]} \ep$.
Applying Lemma \ref{lemma_trace} we see that
\begin{align*}
\left\|w_\ds - i\(i_{\tf}^*(g_{\ep}(w_\ds))\)\right\|_{L^{2(p+1)}(M)}
&\le \left\|w_\ds\right\|_{L^{2(p+1)}(M)} + \left\| i\(i_{\tf}^*(g_{\ep}(w_\ds))\)\right\|_{L^{2(p+1)}(M)}
\\
&\le C\( \ep^{-\frac{N-2s}{4s}} +\left\|g_{\ep}(w_\ds)\right\|_{L^{\frac{4N}{N+6s}}(M)}\)
\\
&\le C\( \ep^{-\frac{N-2s}{4s}} +\ep^{{N-2s \over 8s}}\).
\end{align*}
Using this and the fact that $\ep^{-\ep}= O(1)$, we deduce the desired estimate \eqref{eq_error_0} from \eqref{eq_error_1}.
\end{proof}

\subsection{Linear theory}\label{sec_lin}
To solve \eqref{eq_inter_1}, it is important to understand the linear operator
\begin{equation}\label{eq_lin_op}
L^\ep_\ls(\Phi) := \Phi - (\Pi_\ls^{\ep})^\perp i^*_{\tf}(i(g'_{\ep}(\mw_\els)\Phi)) \quad \text{for } \Phi \in \(K_\ls^{\ep}\)^\perp
\end{equation}
where the function $g_{\ep}$ and $\tf$ are defined in \eqref{eq_gep}.
Letting $\Psi = L^\ep_\ls(\Phi)$, we see that the expression
\begin{equation}\label{eq_lin_op_2}
\begin{cases}
\Phi - i^*_f(i(g'_{\ep}(\mw_\els)\Phi)) = \Psi + \sum_{i=0}^N c_i \mz^i_\els &\text{in } X,\\
\la \Phi, \mz^i_\els \ra_{\tf} = 0 &\text{for all } i = 0, \cdots, N
\end{cases} \end{equation}
with certain pair of constants $(c_0, \cdots, c_N) \in \mr^{N+1}$, is equivalent to \eqref{eq_lin_op}.

This subsection is devoted to deduce that for a fixed $\Psi \in (K_\ls^{\ep})^\perp
$, there are a unique function $\Phi \in (K_\ls^{\ep})^\perp
$ and an $(N+1)$-tuple $(c_0, \cdots, c_N) \in \mr^{N+1}$ satisfying \eqref{eq_lin_op_2}.
This is the content of Proposition \ref{prop_lin_ortho}.
It comes from the fact that the operators $L^\ep_\ls: (K_\ls^{\ep})^\perp \to (K_\ls^{\ep})^\perp$ have the inverses
whose norms are uniformly bounded for $(\ls) \in (\lambda_1^{-1}, \lambda_1) \times M$ and sufficiently small $\ep > 0$ (refer to Lemma \ref{lemma_lin_op_bdd}).

\medskip
We start the proof by showing the \textit{almost orthogonality} of $\mz^i_\ds$'s with respect to the inner product $\la \cdot, \cdot \ra_f$.
As before, we use $\delta = \ep^{\alpha} \lambda$.
\begin{lemma}\label{lemma_lin_ortho}
For each $i, j \in \{0, \cdots, N\}$, we have
\begin{equation}\label{eq_lin_ortho}
\la \mz^i_\ds, \mz^j_\ds \ra_{\tf} = {1 \over \delta^2} \(\beta_i \delta_{ij} + o(1)\) \quad \text{as } \ep \to 0
\end{equation}
where $\beta_i > 0$.
\end{lemma}
\begin{proof}
Recalling that $Z_1^i$'s are solutions of \eqref{eq_lin_pro}, we compute with estimates \eqref{eq_decay_2} and \eqref{eq_decay_3} that
\begin{align*}
\delta^2 \la \mz^i_\ds, \mz^j_\ds \ra_{\tf} &= \kappa_s \delta^2 \int_X \(\rho^{1-2s} \(\nabla \mz^i_\ds, \nabla \mz^j_\ds\)_{\bg} + E(\rho) \mz^i_\ds \mz^j_\ds\) dv_{\bg} + \delta^2 \int_M \tf \mz^i_\ds \mz^j_\ds dv_{\hh}\\
&= \(\int_{\mr^{N+1}_+} t^{1-2s} \nabla Z_1^i \cdot \nabla Z_1^j dxdt + o(1)\) + O\(\delta^2\) + O\(\delta^{2s}\)\\
&= p\int_{\mr^N} w_1^{p-1}z^i_1z^j_1 dx + o(1),
\end{align*}
which implies \eqref{eq_lin_ortho}.
\end{proof}

\noindent From the above lemma and the nondegeneracy result of \cite{DDS} described in Subsection \ref{subsec_trace_ineq}, the following invertibility result of the linear operator $L^\ep_\ls$ can be deduced.
\begin{lemma}\label{lemma_lin_op_bdd}
Suppose that $N > 2s$, $(\ls) \in (\lambda_1^{-1}, \lambda_1) \times M$ and $\ep > 0$ is small enough.
Then there exists a constant $C > 0$ independent of the choice of $(\ls)$ and $\ep$ such that
\begin{equation}\label{eq_lin_op_est}
\|L^\ep_\ls(\Phi)\|_{\tf, \ep} \ge C \|\Phi\|_{\tf, \ep}
\end{equation}
for all $\Phi \in (K_\ls^{\ep})^\perp$.
\end{lemma}
\begin{proof}
We only inspect the case when $g_{\ep}(u) = u_+^{p \pm \ep}$ (and $\tf = f$).
The other case, namely, when $g_{\ep}(u) = u_+^p - \ep fu$ (and $\tf = 0$) is covered in a parallel way.

\medskip
Assume that \eqref{eq_lin_op_est} does not hold so that there are sequences $\ep_n \to 0$, $\lambda_n \to \lambda_{\infty} \in [\lambda_1^{-1}, \lambda_1]$,
$\delta_n = \ep_n^{\alpha}\lambda_n$, $\sigma_n \to \sigma_{\infty} \in M$,
$\Phi_n \in (K_{\lambda_n, \sigma_n}^{\ep_n})^\perp$ and $\Psi_n = L^{\ep_n}_{\lambda_n, \sigma_n}(\Phi_n)$ with
\begin{equation}\label{eq_lin_op_est_1}
\|\Psi_n\|_{f,\ep} \to 0 \quad \text{and} \quad  \|\Phi_n\|_{f,\ep} = 1 \quad \text{as } n \to \infty.
\end{equation}
We may further assume that $\sigma_{\infty} = 0$ by identifying a neighborhood of $\sigma_{\infty}$ in $M$ and that of the origin in $\mr^N$.
According to \eqref{eq_lin_op_2} and Lemma \ref{lemma_lin_ortho}, it is true that
\[- \delta_n^2 (p \pm \ep) \int_M \mw_\dsn^{p-1 \pm \ep} \mz^j_\dsn \Phi_n dv_{\hh} = \delta_n^2 \la \Psi_n, \mz^j_\dsn \ra_f + \sum_{i=0}^N (c_i)_n \(\beta_i \delta_{ij} + o(1)\)\]
for each $j = 0, \cdots, N$.
Following the assertion in the proof of Lemma \ref{lemma_error_est},
it is possible to regard $\Phi_n$ as a function in $\mr^{N+1}_+$ whose support is included in the small half ball
$\overline{B_{\bg}^+(\sigma_n,3r_0)} \subset \overline{B_{\bg}^+} := \overline{B_{\bg}^+(0,4r_0)}$ satisfying $\|\Phi_n\|_{f} \le C_1$ for a fixed constant $C_1 > 0$.
We define $\wtp_n (z) = \delta_n^{N-2s \over 2} \Phi_n(\delta_n x+\sigma_n,\delta_n t)$ for all $z \in \mr^{N+1}_+$.
Then as in \eqref{eq_scale_inv}, one can check that $\|\wtp_n\|_{\dmr}$ is bounded in $n \in \mathbb{N}$ and in particular $\wtp_n \rightharpoonup \wtp_{\infty}$ weakly in $\dmr$.
Hence the compactness property of the trace operator tells us that $\wtp_n \to \wtp_{\infty}$ strongly in $L_{\text{loc}}^q(\mr^N)$ for any $q < {2N \over N-2s}$ and so
\[- \delta_n^2 (p \pm \ep) \int_M \mw_\dsn^{p-1 \pm \ep} \mz^j_\dsn \Phi_n dv_{\hh} = - \delta_n \(\int_{\mr^N}  p w_1^{p-1} z^j_1 \wtp_{\infty} dx + o(1)\) = o(\delta_n).\]
Here the second equality holds, for the assumption $\Phi_n \in (K_{\lambda_n, \sigma_n}^{\ep_n})^\perp$ gives
\begin{equation}\label{eq_lin_op_est_2}
\begin{aligned}
0 = \delta_n \la \Phi_n, \mz^j_\dsn \ra_f &= \delta_n \kappa_s \int_X \rho^{1-2s} (\nabla \mz^j_\dsn, \nabla \Phi_n)_{\bg} dv_{\bg} + O\(\delta_n^{2s}\)\\
&= \int_{\mr^{N+1}_+} t^{1-2s} \nabla Z^j_1 \cdot \nabla \wtp_{\infty} dx dt + o(1)
= \int_{\mr^N} p w_1^{p-1} z^j_1 \wtp_{\infty} dx + o(1).
\end{aligned}
\end{equation}
Since $\left|\delta_n^2 \la \Psi_n, \mz^j_\dsn \ra_f \right| = o(\delta_n)$ by \eqref{eq_lin_op_est_1}, it follows that
\begin{equation}\label{eq_lin_1}
|(c_i)_n| = o(\delta_n) \quad \text{and} \quad \left\|\sum_{i=0}^N (c_i)_n \mz^i_\dsn\right\|_f = o(1).
\end{equation}
Therefore, if we define $\Xi_n(z) = \delta_n^{-{N-2s \over 2}} \Xi(\delta_n^{-1}(x-\sigma_n),\delta_n^{-1}t)$ for any function $\Xi \in C^{\infty}_c(\mr^{N+1})$
and regard it as a function in the open half ball $B_{\bg}^+ \subset X$, which is possible for $n \in \mathbb{N}$ large enough, we see
\begin{multline*}
\kappa_s \int_{B_{\bg}^+} \left[t^{1-2s} (\nabla \Phi_n, \nabla \Xi_n)_{\bg} + E(t)\Phi_n \Xi_n \right] \sqrt{|\bg|} dxdt + \int_{B_{\hh}} \left[f - (p \pm \ep) \mw_\dsn^{p-1 \pm \ep}\right] \Phi_n \Xi_n \sqrt{|\hh|} dx \\
= \la\Psi_n + \sum_{i=0}^N (c_i)_n \mz^i_\dsn, \Xi_n \ra_f = o(1)
\end{multline*}
where $B_{\hh} := B_{\hh}(0, 4r_0) \subset \mr^N$.
Note that $\{\|\Xi_n\|_f\}_{n=1}^{\infty}$ is bounded and that \eqref{eq_bdry_est} implies
\begin{align*}
\int_{B_{\bg}^+} |E(t)||\Phi_n| |\Xi_n| dxdt &\le C \int_{B_{\bg}^+} t^{-2s}|\Phi_n| |\Xi_n| dxdt
\le C \(\int_{B_{\bg}^+} t^{-2s} \Phi_n^2 dxdt\)^{1 \over 2}\(\int_{B_{\bg}^+} t^{-2s} \Xi_n^2 dxdt\)^{1 \over 2}\\
&\le C \|\Phi_n\|_{f} \cdot \delta^{1 \over 2} \(\int_{\mr^{N+1}_+} t^{-2s} \Xi^2 dxdt\)^{1 \over 2} = o(1)
\end{align*}
for $s \in (0,1/2)$, while it remains to hold that $\int_{B_{\bg}^+} |E(t)||\Phi_n| |\Xi_n| dxdt = o(1)$ when $s \in [1/2, 1)$ and $H = 0$ by a similar reasoning.
Hence by taking $n \to \infty$, we obtain from Lemma \ref{lemma_metric_exp} that
\[\kappa_s \int_{\mr^{N+1}_+} t^{1-2s} \nabla \wtp_{\infty} \cdot \nabla \Xi dxdt = p \int_{\mr^N} w_1^{p-1} \wtp_{\infty} \Xi dx,\]
which means that $\wtp_{\infty}$ is a weak solution of \eqref{eq_lin_pro}.
On the other hand, the $\dmr$-norm of $\wtp_{\infty}$ is finite, so the Moser iteration argument works and it reveals that $\wtp_{\infty}$ is $L^{\infty}(\mr^N)$-bounded (see the proof of Lemma 5.1 in \cite{CKL}).
Thus with \eqref{eq_lin_op_est_1} the linear nondegeneracy result in \cite{DDS}, touched in Subsection \ref{subsec_trace_ineq}, implies $\wtp_{\infty} = 0$ in $\mr^N$.
Now we have that
\[\int_{B_{\hh}} \mw_\dsn^{p-1 \pm \ep} \Phi_n^2 \sqrt{|\hh|} dx = \delta_n^{\mp\({N-2s \over 2}\)\ep} \int_{\mr^N} \chi_1^{p-1 \pm \ep}(\delta_n x) w_1^{p-1 \pm \ep}(x) \wtp_n^2(x) \sqrt{|\hh|}(\delta_n x + \sigma_n) dx = o(1). \]
Putting $\Phi = \Phi_n$ into \eqref{eq_lin_op_2} shows then
\[\|\Phi_n\|_{f} = (p \pm \ep) \int_{B_{\hh}} \mw_\dsn^{p-1 \pm \ep} \Phi_n^2 \sqrt{|\hh|} dx + \la\Psi_n + \sum_{i=0}^N (c_i)_n \mz^i_\dsn, \Phi_n \ra_{f} = o(1),\]
and particularly $\|\Phi_n\|_{L^{p+1}(M)} = o(1)$.
At this point, we claim that $\|\Phi_n \|_{L^{q_{\ep}(M)}} = o(1)$.
Once we verify it, together the previous estimate, it will yield that $\|\Phi_n\|_{f,\ep} \to 0$ as $n \to \infty$.
Therefore we will reach a contradiction and our desired inequality \eqref{eq_lin_op_est} should have the validity.
Since the assertion clearly holds in the subcritical or critical cases, it suffices to consider the supercritical case only.
In this situation, by applying Lemma \ref{lemma_trace} and using \eqref{eq_lin_op_2}, \eqref{eq_lin_op_est_1} and \eqref{eq_lin_1}, we get
\begin{equation}\label{eq_lin_3}
\begin{aligned}
\left\| \Phi_n\right\|_{L^{q_{\ep}}(M)} &\le \left\|  i^*_f\(i\(g'_{\ep}\(\mw_\dsn\)\Phi_n\)\) \right\|_{L^{q_{\ep}}(M)} + \left\| \Phi_n - i^*_f\(i\(g'_{\ep}\(\mw_\dsn\)\Phi_n\)\) \right\|_{L^{q_{\ep}}(M)}
\\
&\le \left\|i\(g'_{\ep}\(\mw_\dsn\)\Phi_n\) \right\|_{L^{\frac{Nq_{\ep}}{N+2sq_{\ep}}}(M)} + o(1).
\end{aligned}
\end{equation}
According to H\"older's inequality,
\begin{equation}\label{eq_lin_2}
\left\| i\(g'_{\ep}\(\mw_\dsn\)\Phi_n\) \right\|_{L^{\frac{Nq_{\ep}}{N+2sq_{\ep}}}(M)} \le \left\| g'_{\ep}(w_\dsn) \right\|_{L^{\tilde{r}_{\ep}}(M)} \|\Phi_n\|_{L^{p+1}(M)},
\end{equation}
where $\frac{1}{\tilde{r}_{\ep}}+ \frac{1}{p+1} = \frac{N+2sq_{\ep}}{Nq_{\ep}}$.
Since $\tilde{r}_{\ep} = \frac{N}{2s}+O(\ep)$, we have $\left\| g'_{\ep}(w_\dsn) \right\|_{L^{\tilde{r}_{\ep}}(M)} = O(1)$.
Thus we get from  \eqref{eq_lin_2} that $\left\| i\(g'_{\ep}\(\mw_\dsn\)\Phi_n\) \right\|_{L^{\frac{Nq_{\ep}}{N+2sq_{\ep}}}(M)} = o(1)$, which gives $\|\Phi_n \|_{L^{q_{\ep}(M)}} = o(1)$ with \eqref{eq_lin_3}.
\end{proof}

\noindent As a result, we can construct a solution of \eqref{eq_lin_op_2}.
\begin{prop}\label{prop_lin_ortho}
Given $N > 2s$, fix a point $(\ls) \in (\lambda_1^{-1}, \lambda_1) \times M$ and a small parameter $\ep > 0$ such that Lemma \ref{lemma_lin_op_bdd} holds.
For each $\Psi \in (K_\ls^{\ep})^\perp$, there exists a unique solution $(\Phi, (c_0, \cdots, c_N)) \in (K_\ls^{\ep})^\perp \times \mr^{N+1}$
to equation \eqref{eq_lin_op_2} such that estimate \eqref{eq_lin_op_est} is satisfied.
\end{prop}
\begin{proof}
Firstly let us show that the linear map $L^\ep_\ls$ on $\mh_{\ep}$ is the sum of the identity and a compact operator,
that is to say, the map $\Phi \mapsto (\Pi_\ls^{\ep})^\perp i^*_{\tf}(i(g'_{\ep}(\mw_\els)\Phi))$ for $\Phi \in \mh_{\ep}$ is compact.
Denote $\zeta_3 = {2N^2 \over N^2+4s^2}$.
Then, by Corollary \ref{cor_adj}, we observe that $i^*_{\tf}: L^{\zeta_3}(M) \to \mh_{\ep}$ is a compact operator given $\ep > 0$ small.
Furthermore, since $i(\mw_{\els})$ is in $L^{\infty}(M)$, it holds that $i(g'_{\ep}(\mw_\els)\Phi) \in L^{\zeta_3}(M)$ for any $\Phi \in \mh_{\ep}$.
Consequently, our assertion is true and the proposition follows from a standard argument utilizing the previous lemma and the Fredholm alternative.
\end{proof}

\subsection{Derivation of a solution to the intermediate problem}\label{subsec_inter}
From the unique existence result for the linear problem \eqref{eq_lin_op_2} stated in Proposition \ref{prop_lin_ortho},
we are now able to derive that \eqref{eq_inter_1} is solvable for any given $(\ls) \in (\lambda_1^{-1}, \lambda_1) \times M$ provided $\ep > 0$ sufficiently small.
Let us rewrite problem \eqref{eq_inter_1} as
\begin{multline}\label{eq_inter_3}
L^\ep_\ls(\Phi) = - E^\ep_\ls + N^\ep_\ls(\Phi) := -\(\Pi_{\ls}^{\ep}\)^{\perp}\(\mw_\els - i_{\tf}^*\(g_{\ep}(\mw_\els)\)\) \\
+ \(\Pi_{\ls}^{\ep}\)^{\perp} \(i_{\tf}^*\(g_{\ep}(\mw_\els + \Phi) - g_{\ep}(\mw_\els) - g'_{\ep}(\mw_\els)\Phi\)\).
\end{multline}
\begin{prop}\label{prop_Phi}
Under the assumption of Proposition \ref{prop_lin_ortho} equation \eqref{eq_inter_3} possesses a unique solution $\Phi = \Phi_\els \in (K_{\ls}^{\ep})^{\perp}$ such that
\begin{equation}\label{eq_phi_norm}
\|\Phi_\els\|_{\tf, \ep} = O\(\ep^{\gamma}\)
\end{equation}
where the exponent $\gamma$ is defined in \eqref{eq_error_est_0}.
\end{prop}
\begin{proof}
We define an operator $T_{\ls}^{\ep}: (K_{\ls}^{\ep})^{\perp} \to (K_{\ls}^{\ep})^{\perp}$ by
\[T_{\ls}^{\ep}(\Phi) = \(L_{\ls}^{\ep}\)^{-1}\(- E^\ep_\ls + N^\ep_\ls(\Phi)\).\]
A direct computation using Lemmas \ref{lemma_trace} and \ref{lemma_error_est} shows that it is a contraction map on the set
\[\mathcal{B} = \left\{\Phi \in (K_{\ls}^{\ep})^{\perp}: \|\Phi\|_{\tf,\ep} \le M \ep^{\gamma} \right\} \quad \text{for some large } M > 0.\]
Consequently, it admits a unique fixed point $\Phi_\els \in \mathcal{B}$, which becomes a solution to \eqref{eq_inter_3}.
This completes the proof.
\end{proof}

\section{Finite dimensional reduction}\label{sec_red}
We keep using notations $g_{\ep}(u)$, $\tf$ in \eqref{eq_gep}.
Define also $G_{\ep}(u) = \int_0^t g_{\ep}(t)dt$.

\medskip
It is notable that equations \eqref{eq_main_12}-\eqref{eq_main_21} have the variational structure.
In other words, $U \in \mh_{\ep}$ is a weak solution of \eqref{eq_main_12}-\eqref{eq_main_21} if it is a critical point of the energy functional
\[I_{\ep}(U) := {\kappa_s \over 2}\int_X \(\rho^{1-2s}|\nabla U|_{\bg}^2 + E(\rho)U^2\) dv_{\bg} + {1 \over 2}\int_M \tf U^2 dv_{\hh} - \int_M G_{\ep}(U) dv_{\hh}\]
where $dv_{\bg}$ and $dv_{\hh}$ denote the volume forms on $(X, \bg)$ and its boundary $(M, \hh)$, respectively.
Based on the previous observations, we define a reduced energy functional by
\begin{equation}\label{eq_Jep}
J_{\ep}(\ls) = I_{\ep}\(\mw_\els + \Phi_\els\)
\end{equation}
for any $(\ls) \in (0,\infty) \times M$ where the exponent $\alpha > 0$ is determined in \eqref{eq_alpha} and $\Phi_\els$ denotes the function determined in Proposition \ref{prop_Phi}.

\medskip
The next proposition claims that the well-known finite dimension reduction procedure is still applicable in our setting.
\begin{prop}\label{prop_red}
Assume that $\ep > 0$ is small enough.
Then the reduced energy $J_{\ep}: (0,\infty) \times M \to \mr$ is continuously differentiable.
Moreover, if $J_{\ep}'(\lambda_{\ep}, \sigma_{\ep}) = 0$ for some element $(\lambda_{\ep}, \sigma_{\ep}) \in (0,\infty) \times M$,
then the function $\mw_\elep + \Phi_\elep$ solves problems \eqref{eq_main_12}-\eqref{eq_main_21} (according to the choice of the nonlinearity $g_{\ep}$).
Its trace on $M$ is in $C^{1,\beta}(M)$ for some $\beta \in (0,1)$ determined by $N$ and $s$.
\end{prop}
\begin{proof}
Fix $\ep > 0$ and define a linear operator
\[\mathcal{L}^{\ep}((\ls),U) = U + \(\Pi_\ls^{\ep}\)^{\perp}\left[\mw_{\els} - i^*_{\tf}(i(g_{\ep}(\mw_{\els}+U))) \right]\]
for $((\ls),U) \in (0,\infty) \times M \times \mh_{\ep}$.
Then $\mathcal{L}^{\ep}((\ls),\Phi_{\els}) = 0$ and
\[{\pa \mathcal{L}^{\ep} \over \pa U}((\ls),U) = U - \(\Pi_\ls^{\ep}\)^{\perp}\left[i^*_{\tf}\(i\(g_{\ep}'\(\mw_{\els}\)U\)\) \right].\]
By elliptic regularity, $i(g_{\ep}'(\mw_{\els})\Phi_{\els}) \in L^q(M)$ for some $q > {2N \over N+2s}$. (Refer to the latter part of this proof.)
Hence we know from Corollary \ref{cor_adj} that ${\pa \mathcal{L}^{\ep} \over \pa U}((\ls),\Phi_{\els}): \mh_{\ep} \to \mh_{\ep}$ is a Fredholm operator of index 0.
Moreover, using \eqref{eq_phi_norm}, one can check that it is also injective.
Therefore ${\pa \mathcal{L}^{\ep} \over \pa U}((\ls),\Phi_{\els})$ is invertible and the implicit function theorem shows that the mapping $(\ls) \in (0,\infty) \times M \mapsto \Phi_\els \in \mh_{\ep}$ is $C^1$.
This leads that $J_{\ep}$ is a $C^1$ map.
Furthermore it is a standard step to show that $J_{\ep}'(\lambda_{\ep}, \sigma_{\ep}) = 0$ implies $I_{\ep}'\(\mw_\elep + \Phi_\elep\) = 0$.

In the rest of the proof, we take account of equations \eqref{eq_main_12}.
The other equation \eqref{eq_main_21} can be dealt with similarly. One has then
\begin{multline*}
{\kappa_s \over 2}\int_X \left[\rho^{1-2s} \(\nabla\(\mw_\elep + \Phi_\elep\), \nabla \Xi\)_{\bg} + E(\rho)\(\mw_\elep + \Phi_\elep\)\Xi\right] dv_{\bg}\\
+ \int_M f\(\mw_\elep + \Phi_\elep\)\Xi dv_{\hh} = (p \pm \ep) \int_M \(\mw_\elep + \Phi_\elep\)_+^{p-1 \pm \ep}\Xi dv_{\hh}
\end{multline*}
for any $\Xi \in \hx$.
Putting $\Xi = (\mw_\elep + \Phi_\elep)_-$ into the above identity and then applying \eqref{eq_coercive} verifies that $\mw_\elep + \Phi_\elep \ge 0$ in $\overline{X}$.
By Propositions \ref{prop_ext} and \ref{prop_ext_2}, equation \eqref{eq_ext_1} is solved by the nonnegative function $U = (\rho^*/\rho)^{\verb"z"-N}(\mw_\elep + \Phi_\elep)$ defined in $X$ and its trace $u \ge 0$ on the boundary $M$.
Also, $U$ is not identically zero since $\|\mw_\elep + \Phi_\elep\|_f \ge \|\mw_\elep\|_f - \|\Phi_\elep\|_f \ge C + O(\ep^{\gamma}) > 0$, and it is strictly positive in $X$,
for \eqref{eq_ext_1} is a uniformly elliptic equation in divergence form away from the boundary.
Suppose now that $u(z_0) = 0$ for a point $z_0 \in M$.
Then by the Hopf lemma \cite[Theorem 3.5]{GQ}, we have $(\rho^*)^{1-2s} \pa_{\rho^*} U > 0$ at $z_0$, while \eqref{eq_ext_2} gives
\[\kappa_s (\rho^*)^{1-2s} \pa_{\rho^*} U = - \pa_{\nu}^s U = Q_{\hh}^s u - P_{\hh}^s u = \(Q_{\hh}^s + f - u^{p \pm \ep}\)u = 0 \quad \text{at } z_0.\]
Therefore a contradiction arises and the functions $U$ and $\mw_\elep + \Phi_\elep$ should be strictly positive in $\overline{X}$.

Finally, if the nonlinearity of the problem is subcritical, then \cite[Theorem 3.4]{GQ} implies that $U$ is a locally bounded function in $X$.
Then the regularity property of $\left.\mw_\elep + \Phi_\elep\right|_{M}$ follows directly by the result of \cite[Proposition 3.2]{GQ}.
If our problem is critical or supercritical one, then the nonlinear term is given by $g_{\ep}(u) = u_+^{\frac{N+2s}{N-2s}+\ep} = u^{\frac{4s}{N-2s}+\ep} \cdot u$ for $\ep \ge 0$.
Note that
\[\frac{N}{2s} \cdot \( \frac{4s}{N-2s}+\ep\) = \frac{2N}{N-2s} +\frac{N}{2s}\ep = q_{\ep}\]
(see \eqref{eq_q_ep}).
Therefore $u \in L^{q_{\ep}}(M)$ means that $u^{\frac{4s}{N-2s}+\ep} \in L^{\frac{N}{2s}}(M)$,
and so one can modify the proof of Lemmas 3.5 and 3.8 in \cite{CKL} slightly to show that $U$ is $L^{\infty}$-bounded.
The regularity of $U$ again follows from \cite[Proposition 3.2]{GQ} now.
\end{proof}

\section{Energy expansion}\label{sec_energy}
\subsection{The $C^0$-estimates}\label{subsec_energy}
We set $d_0 = \int_{\mr^N} w_1^{p+1}dx$, $d_1 = \int_{\mr^N} w_1^2dx$, $d_2 = \int_{\mr^N} w_1^{p+1}\log w_1dx$
and $d_s = \kappa_s \int_{\mr^{N+1}_+} t^{2-2s}|\nabla W_1|^2dxdt$ (whose finiteness for $N > 2s+1$ is guaranteed by \eqref{eq_decay_1}).
Then the following asymptotic expansion is valid.

\begin{prop}\label{prop_energy}
Suppose that $\ep > 0$ is sufficiently small, and $H = 0$ if $s \in [1/2,1)$.
In addition, we remind the reduced energy functional $J_{\ep}$  defined in \eqref{eq_Jep}.

\medskip \noindent
(i) Assume further that $N > \max\{4s, 1\}$. If problems \eqref{eq_main_12} is concerned for $s \in (0, 1)$, then it holds
\begin{multline}\label{eq_energy_est_res_1}
J_{\ep}(\ls) = {sd_0 \over N} \pm {\ep \over p+1} \left[\left\{\({N-2s \over 4s}\) \log\ep - {1 \over p+1} \right\} d_0 - d_2\right]\\
+ \ep\left[{d_1 \over 2}f(\sigma)\lambda^{2s} \pm {(N-2s)^2d_0 \over 4N} \log\lambda + o(1)\right].
\end{multline}

\noindent (ii) Let us consider equation \eqref{eq_main_21} under the assumption that $s \in (0,1/2)$ and $N \ge 2$.
Then it follows that
\begin{equation}\label{eq_energy_est_res_3}
J_{\ep}(\ls) = {sd_0 \over N} + \ep^{1 \over 1-2s} \left[{d_1 \over 2}f(\sigma)\lambda^{2s} + \({2N(N-1) + \(1-4s^2\) \over 4N(1-2s)}\) d_s H(\sigma)\lambda + o(1) \right].
\end{equation}

\noindent In the above estimates, $o(1)$ tends to 0 uniformly for $(\ls) \in (\lambda_1^{-1}, \lambda_1) \times M$.
\end{prop}

\noindent To prove this, we need the following lemmas.
\begin{lemma}\label{lemma_appen_0}
Fix any small $\lambda_1 \in (0,1)$. Given $\delta = \ep^{\alpha}\lambda$, we have
\begin{equation}\label{eq_appen_01}
J_{\ep}(\ls) = I_{\ep}\(\mw_{\ds}\) +
\begin{cases}
o(\ep) &\text{for problems } \eqref{eq_main_12},\\
o\(\ep^{1 \over 1-2s}\) &\text{for problem } \eqref{eq_main_21},
\end{cases} \end{equation}
uniformly for $(\ls) \in (\lambda_1^{-1}, \lambda_1) \times M$. \end{lemma}
\begin{proof}
Putting $\Phi_{\ds}$ into \eqref{eq_inter_1} and then applying $\Phi_{\ds} \in (K_{\ls}^{\ep})^{\perp}$ and Taylor's theorem, we obtain
\begin{align*}
&\ J_{\ep}(\ls) - I_{\ep}\(\mw_{\ds}\)\\
&= \la \mw_{\ds} + \Phi_{\ds}, \Phi_{\ds} \ra_{\tf} - \int_M \(G_{\ep}\(\mw_{\ds} + \Phi_{\ds}\) - G_{\ep}\(\mw_{\ds}\)\)\\
&= \int_M \(g_{\ep}\(\mw_{\ds} + \Phi_{\ds}\) - g_{\ep}\(\mw_{\ds}\)\)\Phi_{\ds}\\
&\hspace{150pt} - \int_M \(G_{\ep}\(\mw_{\ds} + \Phi_{\ds}\) - G_{\ep}\(\mw_{\ds}\) - g_{\ep}\(\mw_{\ds}\)\Phi_{\ds}\)\\
&= O\(\|\Phi_{\ds}\|_{\tf}^2\).
\end{align*}
Therefore the conclusion follows by Proposition
\ref{prop_Phi}.
\end{proof}

\begin{lemma}\label{lemma_appen_1}
Suppose that $s \in (0, 1/2)$ and $N > 2s+1$. Then
\begin{equation}\label{eq_appen_11}
\int_{\mr^{N+1}_+} t^{2-2s} |\nabla W_1|^2 dxdt = {4 \over 1+2s} \int_{\mr^{N+1}_+} t^{2-2s} |\nabla_x W_1|^2 dxdt = {1-2s \over 2} \int_{\mr^{N+1}_+} t^{-2s} W_1^2 dxdt < \infty.
\end{equation}
\end{lemma}
\begin{proof}
The argument we will use here is based on the proof of Lemma 7.2 in \cite{GQ}.
We will only prove the first identity, because the second identity can be justified in a similar manner.

If we denote the Fourier transform of $W_1$ with respect to the $x$-variable by $\ww_1$,
then we have $\ww_1(\xi,t) = \hat{w}_1(\xi)\phi(2\pi|\xi|t)$ where $\phi(t)$ is a solution of the equation
\begin{equation}\label{eq_phi}
\phi''(t)+{1-2s \over t}\phi'(t)-\phi(t) = 0 \quad \text{in } \mr_+, \ \phi(0) = 1, \ \lim_{t \to \infty}\phi(t) = 0.
\end{equation}
Thus we have
\begin{equation}\label{eq_W_est_1}
\begin{aligned}
\int_{\mr^{N+1}_+} t^{2-2s} |\nabla_x W_1|^2 dxdt
&= \int_{\mr^{N+1}_+} (2\pi|\xi|)^2 t^{2-2s} \big|\ww_1(\xi,t)\big|^2 d\xi dt\\
&= \int_{\mr^N}(2\pi|\xi|)^{2s-1}|\hat{w}_1(\xi)|^2 d\xi \cdot \int_0^{\infty} t^{2-2s} |\phi(t)|^2 dt
\end{aligned}
\end{equation}
and
\begin{equation}\label{eq_W_est_2}
\begin{aligned}
\int_{\mr^{N+1}_+} t^{2-2s} (\pa_t W_1)^2 dxdt
&= \int_{\mr^{N+1}_+} (2\pi|\xi|)^2 t^{2-2s} |\hat{w}(\xi)|^2|\phi(2\pi|\xi|t)|^2 d\xi dt\\
&= \int_{\mr^N}(2\pi|\xi|)^{2s-1}|\hat{w}_1(\xi)|^2 d\xi \cdot \int_0^{\infty} t^{2-2s} |\phi'(t)|^2 dt.
\end{aligned}
\end{equation}

Since $\phi(t) = 2^{1-s}t^sK_s(t)/\Gamma(s)$ where $K_s$ is the modified Bessel function of the second kind (see \cite[Lemma 14]{Go} for its derivation), $\phi$ decays exponentially as $t$ goes to $\infty$ and $\phi'(t) \sim t^{-1}$ near 0.
Hence after multiplying \eqref{eq_phi} by $t^{3-2s}\phi'(t)$, which converges to 0 as $t \to 0$, and applying integration by parts, we discover that
\begin{align*}
{3-2s \over 2}\int_0^{\infty} t^{2-2s}\phi^2
&= -(1-2s) \int_0^{\infty}t^{2-2s}(\phi')^2 - \int_0^{\infty}t^{3-2s}\phi'\phi''\\
&= -(1-2s) \int_0^{\infty} t^{2-2s}(\phi')^2 + {3-2s \over 2} \int_0^{\infty} t^{2-2s}(\phi')^2 - {1 \over 2} \left.t^{3-2s}(\phi')^2\right]_{t=0}^{\infty}\\
&= {1+2s \over 2} \int_0^{\infty} t^{2-2s}(\phi')^2.
\end{align*}
Putting this with \eqref{eq_W_est_1} and \eqref{eq_W_est_2} gives the first estimate of \eqref{eq_appen_11}.
\end{proof}

\begin{proof}[Proof of Proposition \ref{prop_energy}]
We will accomplish the proof in 3 steps.
We use the notation $\delta = \ep^{\alpha} \lambda$, $\pa_i = \pa_{x_i}$ for $i = 1, \cdots, N$, and $t$ to denote $\rho$ near the boundary.

\medskip \noindent \textsc{Step 1.}
We initiate the proof by computing $\kappa_s \int_X \rho^{1-2s} |\nabla \mw_{\ds}|_{\bg}^2 dv_{\bg}$. By \eqref{eq_decay_1}, we have
\begin{align*}
&\ \kappa_s \int_X \rho^{1-2s} |\nabla \mw_{\ds}|_{\bg}^2 dv_{\bg} \\
&= \kappa_s \int_{B_{r_0}^+} t^{1-2s} \left[\bg^{ij}\pa_i W_{\delta}(x,t)\pa_jW_{\delta}(x,t) + (\pa_tW_{\delta}(x,t))^2\right] \sqrt{|\bg|}dxdt + \begin{cases}
o(\ep) &\text{for } \eqref{eq_main_12},\\
o\(\ep^{1 \over 1-2s}\) &\text{for } \eqref{eq_main_21},
\end{cases} \end{align*}
where $r_0$ is the small positive number chosen in Section \ref{sec_scheme}.
Also Lemma \ref{lemma_metric_exp} implies that $\bar{g}^{ij}= \delta^{ij} + 2\pi^{ij} t + O(|(x,t)|^2)$ and $\sqrt{|\bar{g}|} = 1- H t + O(|(x,t)|^2)$.
Hence we can compute
\begin{equation}\label{eq_energy_est_6}
\begin{aligned}
&\kappa_s \int_{B_{r_0}^+} t^{1-2s} \left[\bg^{ij}\pa_i W_{\delta}(x,t)\pa_jW_{\delta}(x,t) + (\pa_tW_{\delta}(x,t))^2\right] \sqrt{|\bg|}dxdt
\\
&=\kappa_s \int_{B_{r_0}^+} t^{1-2s}|\nabla W_{\delta}|^2 dxdt
+ \kappa_s \( 2 \pi^{ij}(\sigma) \int_{B_{r_0}^+} t^{2-2s} \pa_i W_{\delta}\pa_j W_{\delta} dxdt - H(\sigma) \int_{B_{r_0}^+} t^{2-2s} |\nabla W_{\delta}|^2 dxdt\) \\
&\quad + \int_{B_{r_0}^+} t^{1-2s} O\(|(x,t)|^2\)|\nabla W_{\delta}|^2 dxdt
\end{aligned} \end{equation}
where the last term of the right-hand side is negligible by \eqref{eq_decay_2}.

On the other hand, since $\pa_iW_1$ is odd in $x_i$ and $\pi^{ij} = \hh^{ik}\pi_{kl}\hh^{lj} = \delta^{ik}\pi_{kl}\delta^{lj} = \pi_{ij}$ so that $\pi^{ij}\delta_{ij} = \pi_{ij}\hh^{ij} = H$ at the point $\sigma$ (for we are using now the normal coordinate of $\hh$ at $\sigma$), it holds
\begin{equation}\label{eq_energy_est_5}
2 \pi^{ij}(\sigma) \int_{\mr^{N+1}_+} t^{2-2s} \pa_i W_1\pa_j W_1 dxdt = {2 \over N} H(\sigma) \int_{\mr^{N+1}_+} t^{2-2s} |\nabla_x W_1|^2 dxdt
\end{equation}
(which is finite provided that $N > 2s+1$).

Having them in mind, we consider problems \eqref{eq_main_12} first.
It would be convenient to divide the cases according to the magnitude of $s$.

\medskip \noindent - If $s \in (0,1/2)$, then \eqref{eq_decay_1} gives us  that $\int_{B_{r_0}^+} t^{2-2s} |\nabla W_{\delta}|^2 dxdt = o(\delta^{2s}) = o(\ep)$.

\medskip \noindent - If $s \in [1/2, 1)$, then we observe that \eqref{eq_energy_est_5} remains valid if we change the domains of integration of the both integrals to the half ball $B^+_{r_0}$.
Thus by the hypothesis that $H = 0$ on $M$ if $s \in [1/2, 1)$, we deduce the second term of the right-hand side of \eqref{eq_energy_est_6} vanishes.

\medskip
For problem \eqref{eq_main_21}, we note that if $N \ge 2$, then $N > 2s + 1$ for $s \in (0, 1/2)$.
Hence \eqref{eq_energy_est_5} is meaningful for this problem.

Now applying Lemma \ref{lemma_decay} and \eqref{eq-bubble-eq}, we deduce that
\begin{equation}
\begin{aligned}\label{eq_energy_est_1}
&\kappa_s \int_X \rho^{1-2s} |\nabla \mw_{\ds}|_{\bg}^2 dv_{\bg}
\\
&= \begin{cases}
\int_{\mr^N} w_1^{p+1} dx + o\(\delta^{2s}\) &\text{for } \eqref{eq_main_12}, \\
\int_{\mr^N} w_1^{p+1} dx + \delta \kappa_s H(\sigma) \cdot \({1+2s-2N \over 2N}\) \int_{\mr^{N+1}_+} t^{2-2s} |\nabla W_1|^2 dxdt + o(\delta) &\text{for } \eqref{eq_main_21}.
\end{cases} \end{aligned} \end{equation}

\medskip \noindent \textsc{Step 2.}
Next, we calculate $\kappa_s \int_X E(\rho) \mw_{\ds}^2 dv_{\bg}$.
Assume that $s \in [1/2, 1)$ and $H = 0$. Then $|E(\rho)| \le C \rho^{1-2s}$, so we get
\begin{equation} \label{eq_energy_est_2}
\kappa_s \int_X |E(\rho)| \mw_{\ds}^2 dv_{\bg} \le C \int_{B_{r_0}^+} t^{1-2s} W_{\delta}^2 dx dt = \begin{cases}
O\(\delta^{\min\{2, N-2s\}}\) &\text{if } N \ne 2s + 2,\\
O\(\delta^2 |\log \delta|\) &\text{if } N = 2s + 2.
\end{cases} \end{equation}
On the other hand, if $s \in (0,1/2)$, then \eqref{eq_E_local_2} shows that $\kappa_s \int_X E(\rho) \mw_{\ds}^2 dv_{\bg} = O(\delta) = o(\delta^{2s})$ so that we can neglect this term if problems \eqref{eq_main_12} is considered.
If $N > 2s + 1$, we have more accurate estimate
\begin{equation} \label{eq_energy_est_22}
\begin{aligned}
\kappa_s \int_X E(\rho) \mw_{\ds}^2 dv_{\bg}
&= \kappa_s \({N-2s \over 2}\) \int_X H \rho^{-2s} \mw_{\ds}^2 dv_{\bg} + \begin{cases}
O\(\delta^{\min\{2, N-2s\}}\) &\text{if } N \ne 2s + 2,\\
O\(\delta^2 |\log \delta|\) &\text{if } N = 2s + 2,
\end{cases} \\
&= \kappa_s \({N-2s \over 2}\) H(\sigma) \delta \int_{\mr^{N+1}_+} t^{-2s} W_1^2 dxdt + o(\delta),
\end{aligned} \end{equation}
which is needed for problem \eqref{eq_main_21}.

\medskip \noindent \textsc{Step 3.}
Finally, we turn to estimate $\int_M G_{\ep}(\mw_{\ds}) dv_{\hh}$.
To deal with whole cases, it suffices to compute that
\[I_1 := \int_M \mw_{\ds}^{p+1} dv_{\hh}, \quad I_2 := \int_M \Bigg({\mw_{\ds}^{p+1 \pm \ep} \over p+1 \pm \ep} - {\mw_{\ds}^{p+1} \over p+1}\Bigg) dv_{\hh} \quad \text{and} \quad I_3 := \int_M f \mw_{\ds}^2 dv_{\hh}.\]
Since $dv_{\hh} = \sqrt{|h|}dx = 1 - {1 \over 6}R_{ij}x_ix_j + O(|x|^3)$, under the assumption that $N > 1$ it is plain to obtain for all $s \in (0,1)$ that
\begin{equation}\label{eq_energy_est_3}
I_1 = \int_{\mr^N} w_1^{p+1} dx + o\(\delta^{\max\{2s, 1\}}\) \quad \text{and}
\quad I_3 = \delta^{2s} \(f(\sigma) \int_{\mr^N} w_1^2 dx + o(1)\).
\end{equation}
Besides one can calculate the integral $I_2$ by applying Taylor's theorem and the expansion $(a\ep)^{b\ep} = 1 + b\ep \log(a\ep) + O(\ep^2|\log\ep|)$ which holds for $a > 0$, $b \in \mr$ and small $\ep > 0$, yielding
\begin{equation}\label{eq_energy_est_4}
\begin{aligned}
I_2 &= \int_{\mr^N} \Bigg({(\lambda \ep^{\alpha})^{\mp\(\frac{N-2s}{2}\)\ep}w_1^{p+1\pm\ep} \over p+1\pm\ep} - {w_1^{p+1} \over p+1}\Bigg)dx + O\(\delta^2|\log \delta|\)\\
&= \pm \ep \left[\frac{1}{p+1} \int_{\mr^N} w_1^{p+1}\log w_1dx - (\alpha \log\ep + \log \lambda) \cdot {(N-2s)^2 \over 4N} \cdot \int_{\mr^N}w_1^{p+1}dx \right]\\
&\quad \mp {\ep \over (p+1)^2} \int_{\mr^N} w_1^{p+1} dx + O\(\ep^2|\log \ep|\) + O\(\delta^2|\log \delta|\).
\end{aligned}
\end{equation}

\medskip
From \eqref{eq_energy_est_1}-\eqref{eq_energy_est_4} and \eqref{eq_appen_01}, estimations \eqref{eq_energy_est_res_1} and \eqref{eq_energy_est_res_3} can be deduced at once.
This concludes the proof.
\end{proof}

\subsection{The $C^1$-estimates}\label{subsec_energy_2}
The aim of this subsection is to improve Proposition \ref{prop_energy} by showing that the $o(1)$-terms go to 0 in $C^1$-sense.
Unfortunately there is some technical difficulty in obtaining the $C^1$-estimates,
because the estimate $\|\Phi_\els\|_{\bar{f}} = O(\ep^{\gamma})$ in \eqref{eq_phi_norm} (and $\|\Phi_\els\|_{L^q (M)}=O(\ep^{\gamma})$ for $1 \le q \le q_{\ep}$) of the remainder term $\Phi_\els$ is not so small
compared with the blow up rate $\ep^{-\alpha}$ of the bubbles $\mw_\els$, especially when $s$ is close to zero.
In fact, the standard argument for the $C^1$-estimates of $J_{\ep}$ (see e.g. \cite{MP}) provides only the bound $O(\ep^{-\alpha+2\gamma})$ for the error term,
which is not tolerated in \eqref{eq-c1-1} and \eqref{eq-c1-3} below.
Nevertheless, we can achieve the $C^1$-estimates by modifying some ideas in Esposito-Musso-Pistoia \cite{EMP}.

\begin{prop}\label{prop_energy_1}
Estimates \eqref{eq_energy_est_res_1} and \eqref{eq_energy_est_res_3} are valid $C^1$-uniformly for $(\ls) \in (\lambda_1^{-1},\lambda_1) \times M$. Precisely, the following holds for each fixed point $\sigma_0 \in M$.
Suppose that $y \in \mr^N$ is a point near the origin.

\medskip \noindent (i) Under the assumption of (i) in Proposition \ref{prop_energy}, we have
\begin{equation}\label{eq-c1-1}
\left.\frac{\pa}{\pa y_k} J_{\ep}(\lambda, \exp_{\sigma_0}(y))\right|_{y=0} = \ep \frac{\pa}{\pa y_k} \left[\frac{d_1}{2} f(\exp_{\sigma_0}(y)) \lambda^{2s}\right]_{y=0} + o(\ep)
\end{equation}
for each $1 \le k \le N$ and
\begin{equation}\label{eq-c1-2}
\frac{\pa}{\pa \lambda}J_{\ep}(\lambda, \sigma) = \ep \left[ d_1s f(\sigma) \lambda^{2s-1} \pm \frac{(N-2s)^2 d_0}{4N} {1 \over \lambda} \right] + o(\ep).
\end{equation}

\medskip \noindent (ii) Under the assumption of (ii) in Proposition \ref{prop_energy}, we have
\begin{multline}\label{eq-c1-3}
\left.\frac{\pa}{\pa y_k} J_{\ep}(\lambda, \exp_{\sigma_0}(y))\right|_{y=0} \\
= \ep^{\frac{1}{1-2s}} \frac{\pa}{\pa y_k}\left[ \frac{d_1}{2} f(\exp_{\sigma_0}(y))\lambda^{2s} + \({2N(N-1) + \(1-4s^2\) \over 4N(1-2s)}\) d_s H(\exp_{\sigma_0}(y))\lambda\right]_{y=0} + o\(\ep^{\frac{1}{1-2s}}\)
\end{multline}
for each $1 \le k \le N$ and
\begin{equation}\label{eq-c1-4}
\frac{\pa}{\pa \lambda} J_{\ep}(\lambda, \sigma)
= \ep^{\frac{1}{1-2s}} \left[ d_1s f(\sigma)\lambda^{2s-1} + \({2N(N-1) + \(1-4s^2\) \over 4N(1-2s)}\) d_s H(\sigma) \right]_{y=0} + o\(\ep^{\frac{1}{1-2s}}\).
\end{equation}
\end{prop}
Let us note that $\frac{\pa}{\pa \lambda} W_{\ep^{\alpha}\lambda}$ is a even function in $x \in \mr^N$ like the bubble $W_{\ep^{\alpha}\lambda}$ and has the same decaying property as $W_{\ep^{\alpha}\lambda}$.
From this fact we can see that all the error estimates in the proof of Proposition \ref{prop_energy} hold exactly in the same manner even if they are differentiated in the $\lambda$-variable.
This tells us that \eqref{eq_energy_est_res_1} and \eqref{eq_energy_est_res_3} hold in $C^1$-sense with respect to $\lambda$, i.e., \eqref{eq-c1-2} and \eqref{eq-c1-4} are true.
Thus it only remains to show that \eqref{eq_energy_est_res_1} and \eqref{eq_energy_est_res_3} also hold in $C^1$-sense with respect to $\sigma$, or equivalently, \eqref{eq-c1-1} and \eqref{eq-c1-3} are valid.

\medskip
We fix $\sigma_0 \in M$ and set $\sigma (y) = \exp_{\sigma_0}(y)$ for $y \in B^N(0,4r_0)$ (recall that $4r_0 > 0$ is selected to be smaller than the injectivity radius of $M$) for conciseness.
For the proof, we first need to establish several preliminary lemmas.
\begin{lemma}\label{lem-yw}
Recall the definition of the truncation function $\chi_1$ which was introduced in \eqref{eq_first_approx},
and the fact that any point $z \in X$ located sufficiently close to $\sigma_0 \in M$ can be described as $z = (\sigma(x), t)$ for some $x \in B^N(0,2r_0)$ and $t \in (0, r_0)$.
Also fix any $1 \le k \le N$.
\begin{enumerate}
\item
For any $z = (\sigma(x), t)$ near the point $\sigma_0$, it holds that
\begin{equation}\label{eq-yw-2}
\({\pa \over \pa y_k} \mw_{\dsy}\)_{y=0} (z) = - \chi_1(|(x,t)|) \pa_k W_{\delta}(x,t) + \varrho_1(\sigma(x),t),
\end{equation}
where $\varrho_1$ is a function on $X$ supported on the half ball $B_{\bg}^+(\sigma_0, 2r_0)$ (defined in \eqref{eq_B_bg}) satisfying $\|\varrho_1\|_{\tf, \ep}= O(\delta)$.
\item For any $z$ near the point $\sigma_0$ and $0 \le i \le N$, we have
\begin{equation}\label{eq-yw-3}
\({\pa \over \pa y_k} \mz_{\dsy}^{i}\)_{y=0}(z) = -\chi_1(|(x,t)|) \pa_k Z^i_{\delta}(x,t) + \varrho_2^i (\sigma(x),t)
\end{equation}
where $\varrho_2^i$ is a function on $X$ supported on $B_{\bg}^+(\sigma_0, 2r_0)$ such that $\|\varrho_2^i\|_{\tf, \ep}= O(1)$.
\end{enumerate}
\end{lemma}
\begin{proof}
Using the chain rule and Lemma \ref{lem-pw-decay}, we compute
\begin{align*}
{\pa \over \pa y_k} \mw_{\ds(y)}(z) &=\chi_1(d (z,\sigma (y))) {\pa  W_{\delta} \over \pa y_k} (\mathcal{E}(y,x), t) +  {\pa \chi_1 \over \pa y_k} (d(z,\sigma (y))) W_{\delta}(\mathcal{E}(y,x),t)
\\
&= \chi_1 (d (z,\sigma (y))) \sum_{j=1}^{N} \left[ \pa_j W_{\delta}(\mathcal{E}(y,x),t) {\pa \mathcal{E}_j(y,x) \over \pa y_k} \right] + O\(\delta^{N-2s \over 2} \cdot {|\nabla \chi_1| (|(x-y,t)|) \over |(x-y,t)|^{N-2s}}\)
\end{align*}
where we set $\mathcal{E}(y,x) = \exp_{\sigma (y)}^{-1} (\sigma(x)) = (\mathcal{E}_1(y,x), \cdots, \mathcal{E}_N(y,x)) \in \mr^N$.
Therefore replacing $(\sigma(x), t)$ with $(\exp_{\sigma(y)}(x), t)$ in the previous inequalities, we obtain
\begin{equation}\label{eq-yw-1}
\begin{aligned}
&{\pa \over \pa y_k} \mw_{\dsy} \( \exp_{\sigma (y)}(x), t\)
\\
&= \chi_1 \(d ((\exp_{\sigma (y)}(x),t),\sigma (y))\) \sum_{j=1}^{N} \left[ \pa_j W_{\delta}(x,t) {\pa \mathcal{E}_j \over \pa y_k} \(y, \sigma^{-1}\(\exp_{\sigma (y)}\)(x)\)\right] + O\(\delta^{N-2s \over 2} {|\nabla \chi_1| (|(x,t)|) \over |(x,t)|^{N-2s}}\).
\end{aligned}
\end{equation}
By (6.12) of \cite{MP}, it holds that
\[\left. {\pa \mathcal{E}_j \over \pa y_k} \(y, \sigma^{-1}\(\exp_{\sigma (y)}\)(x)\) \right|_{y=0} = \frac{\pa \mathcal{E}_j}{\pa y_k} (0,x) = -\delta_{kj} + O\(|x|^2\).\]
Taking $y=0$ on the both sides of \eqref{eq-yw-1} and inserting the above in the result yields
\begin{multline*}
\(\frac{\pa}{\pa y_k} \mw_{\dsy}\)_{y=0} \( \exp_{\sigma_0}(x), t\) \\
= - {\chi_1}(|(x,t)|) \pa_k W_{\delta}(x,t)+ \underbrace{\chi_1(|(x,t)|) \sum_{j=1}^{N} \left[ \pa_j W_{\delta}(x, t) \right] O\(|x|^2\) + O\(\delta^{N-2s \over 2} {|\nabla \chi_1| (|(x,t)|) \over |(x,t)|^{N-2s}}\)}_{:= \varrho_1 (x,t)}.
\end{multline*}
We readily find that $\|\varrho_1\|_{\tf, \ep}= O(\delta)$, and thus arrive at the first equality \eqref{eq-yw-2}.

The same argument can be applied to prove the second equality \eqref{eq-yw-3}.
The proof is completed.
\end{proof}

We remind from Proposition \ref{prop_Phi} that $\Phi_{\els}$ solves equation \eqref{eq_inter_1}. Hence for some constants $c_i \in \mr$, $0\le i \le N$, we have
\begin{equation}\label{eq-phi-c}
\Phi_\els = - \mw_\els + i_{\tf}^*\(i\(g_{\ep}(\mw_\els + \Phi_\els)\)\)  + \sum_{i=0}^{N} c_i \mz_{\els}^{i}.
\end{equation}
\begin{lemma}\label{lem-c}
In \eqref{eq-phi-c}, we have that $c_i = O(\ep^{\gamma +\alpha})$ for each $0 \le i \le N$.
\end{lemma}
\begin{proof}
Fixing any $i \in \{0, \cdots, N\}$ and taking the inner product $\la \cdot , \mz_{\ds}^i \ra_{\tf}$ on \eqref{eq-phi-c}, we get that
\begin{equation}\label{eq-c-1}
\begin{aligned}
&c_i \la \mz_{\ds}^i, \mz_{\ds}^i\ra_{\tf} + \sum_{j \ne i} c_j \la \mz_{\ds}^j, \mz_{\ds}^i\ra_{\tf} = \la \mw_{\ds}, \mz_{\ds}^{i} \ra_{\tf} - \int_{M} g_\ep (\mw_{\ds} + \Phi_{\ds} ) \mz_{\ds}^i
\\
&= \( \la \mw_{\ds}, \mz_{\ds}^{i} \ra_{\tf} - \int_{M} g_{\ep}(\mw_{\ds}) \mz_{\ds}^i\) + \(\int_{M} (g_{\ep}(\mw_{\ds}) - g_{\ep}(\mw_{\ds}+\Phi_{\ds}))\mz_{\ds}^i\)
\end{aligned}
\end{equation}
where $\delta = \ep^{\alpha} \lambda$.
Replacing $\Phi$ by $\mz_{\ds}^i$ in the proof of Lemma \ref{lemma_error_est} and using the estimate $\left\|\mz_{\ds}^i\right\|_{\tf,\ep} = O(\ep^{-\alpha})$ instead of $\|\Phi \|_{\tf} = O(1)$, we may deduce that
\[\la \mw_{\ds}, \mz_{\ds}^{i} \ra_{\tf} - \int_{M} g_{\ep}(\mw_{\ds}) \mz_{\ds}^i = O\(\ep^{\gamma -\alpha}\).\]
Next we apply H\"older's inequality to ascertain
\begin{align*}
\int_{M} (g_{\ep}(\mw_{\ds}) -g_{\ep}(\mw_{\ds} + \Phi_{\ds}))\mz_{\ds}^i &= O \( \left\|g_{\ep}' (\mw_{\ds})\right\|_{L^{\frac{N}{2s}}(M)}
\cdot \left\| \Phi_{\ds}\right\|_{L^{\frac{2N}{N-2s}}(M)}
\cdot \left\| \mz_{\ds}^i\right\|_{L^{\frac{2N}{N-2s}}(M)}\) \\
&= O\(\ep^{\gamma -\alpha}\).
\end{align*}
Combining these two estimates and \eqref{eq_lin_ortho}, we derive from \eqref{eq-c-1} that
\[c_i \ep^{-2\alpha} + \sum_{j \ne i} c_j o\(\ep^{-2\alpha}\) = O\(\ep^{\gamma-\alpha}\),\]
which yields the desired estimate $c_i = O(\ep^{\gamma + \alpha})$.
\end{proof}

Recall a fixed point $\sigma_0 \in M$ and the map $\sigma (y) = \exp_{\sigma_0}(y)$ defined for $y \in B^N(0,4r_0)$.
In the next lemma, we shall replace the derivatives $\pa_{y_k}\mw_{\dsy}$ and  $\pa_{y_k}\Phi_{\dsy}$ with respect to the parameters
by the derivatives $\pa_k\mw_{\dsy}$ and $\pa_k \Phi_{\dsy}$ with respect to the spatial variables in the expression of $\pa_{y_k} J_{\ep}(\lambda, \sigma(y))|_{y=0}$.
This will permit us to take integration by parts to evaluate $\pa_{y_k} J_{\ep}(\lambda, \sigma(y))|_{y=0}$.
This idea was introduced in \cite{EMP} where existence of the bubbling solutions for the two dimensional Lane-Emden-Fowler equation was examined.

Take a cut-off function $\chi_2:(0,\infty) \to [0,1]$ such that $\chi_2=1$ on $(0,2r_0)$ and $0$ on $(4r_0, \infty)$.
Then we see that $\chi_2 = 1$ on supp$(\chi_1)$.
We also set a function $\whp_{\ds}: \overline{\mr^{n+1}_{+}} \to \mr$ by
\[ \whp_{\ds}(x,t) = \chi_2 \left(d((\sigma(x),t), \sigma)\right) \Phi_{\ds} (\sigma(x),t),\]
which satisfies supp$(\whp_{\ds}) \subset B^+_{4r_0}$, and a function $\wtp_{\ds}^k : X \to \mr$ ($k = 1, \cdots, N$) by
\[\wtp_{\ds}^k (z) = \begin{cases}
\(\pa_k \whp_{\ds}\)(x,t) &\text{if } z \in X \text{ is near $M$ so that it can be written as } z = (\exp_{\sigma}(x),t), \\
0 &\text{otherwise}.
\end{cases}\]
Then we have the following result.
\begin{lemma}
We have
\[\left. I_{\ep}'(\mw_{\dsy} + \Phi_{\dsy}) \pa_{y_k} \Phi_{\dsy}\right|_{y=0} = - I_{\ep}'(\mw_{\dsz} + \Phi_{\dsz}) \wtp_{\dsz}^k + O\(\ep^{2\gamma}\)\]
and
\begin{multline*}
\left.I_{\ep}' (\mw_{\dsy} + \Phi_{\dsy})\pa_{y_k} \mw_{\dsy}\right|_{y=0} \\
= - I_{\ep}' (\mw_{\dsz} + \Phi_{\dsz}) \pa_k \(\chi_1(|(x,t)|) W_{\delta}(x,t)\) + \begin{cases}
o(\ep) &\text{for } \eqref{eq_main_12},\\
O\(\ep^{\gamma+\alpha}\) &\text{for } \eqref{eq_main_21},
\end{cases} \end{multline*}
where $z = (\sigma(x), t) \in X$ satisfies $d(z,\sigma_0) = |(x,t)| \le 2r_0$.
\end{lemma}
\begin{proof}
From \eqref{eq-phi-c} and the fact that $\la \mz^i_\ds, \Phi_\ds \ra_{\tf} = 0$ for all $\sigma \in M$, we see that
\begin{equation}\label{eq-I'-yk}
\begin{aligned}
&I_{\ep}' (\mw_{\dsy} + \Phi_{\dsy}) (\pa_{y_k} \Phi_{\dsy})\\
&= \sum_{i=0}^N c_i \la \mathcal{Z}_{\dsy}^{i}, \pa_{y_k} \Phi_{\dsy} \ra_{\tf} = - \sum_{i=0}^{N} c_i \la \pa_{y_k} \mz_{\dsy}^{i}, \Phi_{\dsy} \ra_{\tf}
\\
&= - \sum_{i=0}^{N} c_i \left[ \kappa_s \int_{\mr^{N+1}_{+}} t^{1-2s} \(\nabla \pa_{y_k} \mz_{\dsy}^{i} (\sigma(x),t), \nabla \whp_{\dsy}(x,t)\)_{\bg} \sqrt{|\bg|}dxdt \right. \\
&\hspace{50pt} + \kappa_s \int_{\mr^{N+1}_{+}} E(t) \pa_{y_k} \mz_{\dsy}^{i} (\sigma(x),t)\whp_{\dsy}(x,t) \sqrt{|\bg|}dxdt \\
&\hspace{50pt} \left.+ \int_{\mr^N} \tf(\sigma(x)) \pa_{y_k}\mz_{\dsy}^{i} (\sigma(x),0)\whp_{\dsy}(x,0) \sqrt{|\hh|} dx \right].
\end{aligned} \end{equation}
On the other hand,
\begin{equation}\label{eq-I'-xk}
\begin{aligned}
&I_{\ep}'(\mw_{\dsz} + \Phi_{\dsz}) \wtp_{\dsz}^k = \sum_{i=0}^N c_i \la \mz_{\dsz}^i, \wtp_{\dsz}^k \ra_{\tf}
\\
&= \sum_{i=0}^{N} c_i \left[ \kappa_s \int_{\mr^{N+1}_{+}} t^{1-2s} \(\nabla\(\chi_1(|(x,t)|)Z^i_{\delta}(x,t)\), \nabla\(\pa_k \whp_{\dsz}(x,t)\)\)_{\bg} \sqrt{|\bg|}dxdt \right. \\
&\hspace{40pt} + \kappa_s \int_{\mr^{N+1}_{+}} E(t) \chi_1(|(x,t)|)Z^i_{\delta}(x,t)\pa_k \whp_{\dsz}(x,t) \sqrt{|\bg|}dxdt \\
&\hspace{40pt} \left.+ \int_{\mr^N} \tf(\sigma(x)) \chi_1(|x|) z^i_{\delta}(x) \pa_k \whp_{\dsz}(x,0) \sqrt{|\hh|} dx \right].
\end{aligned} \end{equation}

Let us compare \eqref{eq-I'-yk}, for which $y = 0$ is taken, and \eqref{eq-I'-xk}.
Employing \eqref{eq_E_local_2}, \eqref{eq_bdry_est}, \eqref{eq-yw-3} and the observation that $\pa_k \sqrt{|\bg|}  = O(|(x,t)|)$
which stems from Lemma \ref{lemma_metric_exp}, and applying the integration by parts, we have
\begin{align*}
&\left| \int_{\mr^{N+1}_{+}} E(t) \(\pa_{y_k} \mz_{\dsy}^{i} (\sigma(x),t)\)_{y=0} \whp_{\dsz}(x,t) \sqrt{|\bg|}dxdt \right.
\\
&\quad \left. - \int_{\mr^{N+1}_{+}} E(t) \chi_1(|(x,t)|)Z^i_{\delta}(x,t)\pa_k \whp_{\dsz}(x,t) \sqrt{|\bg|}dxdt \right|
\\
&\le \int_{\mr^{N+1}_{+}} O\(t^{-2s}\) ((|\chi_1| + |\pa_k \chi_1|)(|(x,t)|) + O(|(x,t)|)) \left|Z_{\delta}^i(x,t) \whp_{\dsz}(x,t)\right| dx dt
\\
&\quad + \left| \int_{\mr^{N+1}_{+}} E(t) \varrho_2^i (x,t) \whp_{\dsz}(x,t) \sqrt{|\bg|} dx dt \right|
\\
&\le C \left\|t^{-2s} Z_{\delta}^i \right\|_{L^2 (B^+_{2r_0})} \left\|\Phi_{\dsz} \right\|_{\tf}
+ C \left\|\varrho_2^i\right\|_{\tf} \cdot \left\|\Phi_{\dsz} \right\|_{\tf} 
= O\(\ep^{\gamma-\alpha}\)
\end{align*}
for $s \in (0,1/2)$. If $s \in [1/2,1)$ and $H = 0$, the above term has a better bound $O(\ep^{\gamma})$. Similarly,
\begin{align*}
&\left| \int_{\mr^N} \tf(\sigma(x)) \(\pa_{y_k} \mz_{\dsy}^{i} (\sigma(x),0)\)_{y=0} \whp_{\dsz}(x,0) \sqrt{|\hh|} dx \right. \\
&\quad \left. - \int_{\mr^N} \tf(\sigma(x)) \chi_1(|x|) z^i_{\delta}(x) \pa_k \whp_{\dsz}(x,0) \sqrt{|\hh|} dx \right|
\\
&\le \int_{\mr^N} \left|\pa_k\(\tf(\sigma(x)) \chi_1(|x|) \sqrt{|\hh|}\)\right| \left|z^i_{\delta}(x) \whp_{\dsz}(x,0)\right| dx + \int_{\mr^N} \left|\tf(\sigma(x)) \varrho_2^i (x,0) \whp_{\dsz}(x,0)\right| \sqrt{|\hh|} dx
\\
&\le C \(\left\|z^i_{\delta}\right\|_{L^{\frac{2N}{N+2s}}(B^N(0,4r_0))} \left\|\Phi_{\dsz} \right\|_{L^{\frac{2N}{N-2s}}(M)} + \left\|\varrho_2^i (\cdot,0)\right\|_{L^{\frac{2N}{N-2s}}(M)} \left\|\Phi_{\dsz} \right\|_{L^{\frac{2N}{N-2s}}(M)}\)
\\
&= O\(\ep^{(2s-1)\alpha + \gamma} + \ep^{\gamma}\).
\end{align*}
Finally we use Lemmas \ref{lem-yw}, \ref{lemma_metric_exp} and \ref{lemma_decay} to get 
\begin{align*}
&\left| \int_{\mr^{N+1}_{+}} t^{1-2s} \(\nabla \pa_{y_k} \mz_{\dsy}^{i} (\sigma(x),t), \nabla \whp_{\dsy}(x,t)\)_{\bg} \sqrt{|\bg|}dxdt \right.
\\
&\quad \left. - \int_{\mr^{N+1}_{+}} t^{1-2s} \(\nabla\(\chi_1(|(x,t)|)Z^i_{\delta}(x,t)\), \nabla\(\pa_k \whp_{\dsz}(x,t)\)\)_{\bg} \sqrt{|\bg|}dxdt \right|
\\
&\le \int_{\mr_{+}^{N+1}} t^{1-2s} \left[ \left|\nabla \(\chi_1(|(x,t)|)Z^i_{\delta}(x,t)\) \right| O(|(x,t)|) + \left|\nabla \(\pa_k \chi_1(|(x,t)|) Z_{\delta}^i(x,t)\)\right| \right] \cdot \left|\nabla \whp_{\dsz}(x,t)\right| dxdt
\\
&\quad + \int_{\mr_{+}^{N+1}} t^{1-2s} \left|\nabla \varrho_2^i (x,t) \right| \cdot \left|\nabla \whp_{\dsz}(x,t)\right| \sqrt{|\bg|} dxdt = O\(\ep^{\gamma}\).
\end{align*}

Combining the above three estimates with Lemma \ref{lem-c}, we reach at
\begin{align*}
I_{\ep}' (\mw_{\ds} + \Phi_{\ds}) \(\pa_{y_k} \Phi_{\dsy}\) + I_{\ep}' (\mw_{\ds}+ \Phi_{\ds}) \wtp_{\sigma_0}^k
&= O\(\ep^{\gamma +\alpha}\) \cdot O\(\ep^{\gamma-\alpha} + \ep^{(2s-1)\alpha + \gamma}
\)
= O\(\ep^{2\gamma}\).
\end{align*}
It proves the first identity.

We turn to prove the second identity. For this we apply Lemmas \ref{lemma_decay}, \ref{lem-yw} and \ref{lem-c} to certify
\begin{align*}
&I_{\ep}' (\mw_{\dsy}+ \Phi_{\dsy})\pa_{y_k}\mw_{\dsy} + I_{\ep}' (\mw_{\dsz} + \Phi_{\dsz}) \pa_k \(\chi_1(|(x,t)|) W_{\delta}(x,t)\)\\
& =\sum_{i=0}^N c_i \la \mz_{\dsz}^{i}, \left. \pa_{y_k} \mw_{\dsy} \right|_{y=0} + \chi_1(|(x,t)|) \pa_k W_{\delta}(x,t) + \chi_1'(|(x,t)|) \pa_k|(x,t)| W_{\delta}(x,t) \ra_{\tf}
\\
&= \sum_{i=0}^N c_i \la \mz_{\ds_0}^i, \varrho_1^i + \chi_1'(|(x,t)|) \pa_k|(x,t)| W_{\delta}(x,t) \ra_{\tf}
\\
&= O \(\sum_{i=0}^N |c_i| \left\|\mz_{\ds_0}^i\right\|_{\tf}  \cdot \left\|\varrho_1^i\right\|_{\tf}\) + \begin{cases}
O\(|c_i|\ep^{1-\alpha}|\log\ep|\) &\text{for } \eqref{eq_main_12},\\
o(|c_i|) &\text{for } \eqref{eq_main_21}
\end{cases} \\
&= O\(\ep^{\gamma + \alpha} \ep^{-\alpha}\ep^{\alpha}\) + \begin{cases}
O\(\ep^{1+\gamma}|\log\ep|\) &\text{for } \eqref{eq_main_12},\\
o\(\ep^{\gamma+\alpha}\) &\text{for } \eqref{eq_main_21}
\end{cases}
= \begin{cases}
o(\ep) &\text{for } \eqref{eq_main_12},\\
O\(\ep^{\gamma+\alpha}\) &\text{for } \eqref{eq_main_21}.
\end{cases}
\end{align*}
Here we also used
\[\la \mz_{\ds_0}^i, \chi_1'(|(x,t)|) \pa_k|(x,t)| W_{\delta}(x,t) \ra_{\tf} = O\(\delta^{N-2s-1}|\log \delta|\) 
= \begin{cases} O\(\delta^{2s-1}|\log \delta|\) &\text{if } N > 4s,\\
o(1) &\text{if } N > 2s + 1.
\end{cases}\]
Our assertion is proved.
\end{proof}
Now we are ready to establish the desired $C^1$-estimates of the reduced energy functional $J_{\ep}$.
\begin{proof}[Proof of Proposition \ref{prop_energy_1}]
For the sake of simplicity, we identify $\wtp_{\dsz}^k = \pa_k \whp_{\dsz}$ and use an abbreviation $(\chi_1 \pa_k W_{\delta})(z) = \chi_1(|(x,t)|) W_{\delta}(x,t)$ defined for $z = (\sigma(x), t) \in X$ near $\sigma_0 \in M$.
We may assume that the domain of these functions is the Euclidean space $\mr^{N+1}_+$.
By the previous lemma, we have
\begin{multline*}
\left. I_{\ep}'(\mw_{\dsy} + \Phi_{\dsy})\(\pa_{y_k} \mw_{\dsy} + \pa_{y_k} \Phi_{\dsy}\)\right|_{y=0}
\\
= - I_{\ep}'(\mw_{\ds} + \Phi_{\ds}) \pa_k \(\chi_1W_{\delta} + \whp_{\dsz}\) + \begin{cases}
o(\ep) &\text{for problems } \eqref{eq_main_12},\\
o\(\ep^{\alpha}\) &\text{for problem } \eqref{eq_main_21}.
\end{cases}.
\end{multline*}
Let us decompose
\[I_{\ep}'(\mw_{\dsz} + \Phi_{\dsz}) \pa_k \(\chi_1W_{\delta} + \whp_{\dsz}\)
= I'_1 + I'_2 + I'_3 - I'_4,\]
where
\[I'_1 = {\kappa_s} \int_{\mr^{N+1}_{+}} t^{1-2s} \(\nabla \(\chi_1 W_{\delta} + \Phi_{\dsz}\), \nabla\pa_k \(\chi_1W_{\delta} + \whp_{\dsz}\)\)_{\bg} \sqrt{|\bg|} dx dt,\]
\[I'_2 = \kappa_s \int_{\mr^{N+1}_{+}} E(t) \(\chi_1 W_{\delta} + \Phi_{\dsz}\) \pa_k \(\chi_1W_{\delta} + \whp_{\dsz}\) \sqrt{|\bg|} dx dt,\]
\[I'_3 = \int_{\mr^N} \tf(\sigma(x)) \(\chi_1 w_{\delta} + \Phi_{\dsz}\) \pa_k \(\chi_1w_{\delta} + \whp_{\dsz}\) \sqrt{|\hh|} dx\]
and
\[I'_4 = \int_{\mr^N} g_{\ep}\(\chi_1 w_{\delta} + \Phi_{\dsz}\) \pa_k \(\chi_1w_{\delta} + \whp_{\dsz}\) \sqrt{|\hh|} dx.\]
We will calculate each term to conclude the proof of Proposition \ref{prop_energy_1}.

\medskip \noindent \textsc{1. Estimate of $I'_1$.}
In this step, we only consider problem \eqref{eq_main_21} in order to ensure the finiteness of the value $d_s$ defined in the beginning of Subsection \ref{subsec_energy}.
To handle the other case \eqref{eq_main_12} is an easier task.

Direct computation shows that
\[\int_{\mr_+^{N+1}} t^{1-2s} \(\nabla \((1-\chi_2)\Phi_{\dsz}\), \nabla \pa_k \(\chi_1W_{\delta} + \whp_{\dsz}\)\)_{\bg} \sqrt{|\bg|} dx dt = O\(\ep^{2\gamma}\).\]
Thus we have $I_1' = I_{11}' + I_{12}' + O(\ep^{2\gamma})$ where
\begin{align*}
I_{11}' &= {\kappa_s} \int_{\mr^{N+1}_{+}} t^{1-2s} \bg^{ij} \pa_i\(\chi_1 W_{\delta} + \whp_{\dsz}\) \pa_j\pa_k \(\chi_1W_{\delta} + \whp_{\dsz}\) \sqrt{|\bg|} dx dt
\intertext{and}
I_{12}' &= {\kappa_s} \int_{\mr^{N+1}_{+}} t^{1-2s} \pa_t\(\chi_1 W_{\delta} + \whp_{\dsz}\) \pa_t\pa_k \(\chi_1W_{\delta} + \whp_{\dsz}\) \sqrt{|\bg|} dx dt.
\end{align*}
We shall compute the term $I_{11}'$ first. By \eqref{eq_decay_1}, \eqref{eq_decay_2} and \eqref{eq_phi_norm}, we discover
\begin{equation}\label{eq-m1-1}
\begin{aligned}
I_{11}' &= - {\kappa_s \over 2} \int_{\mr^{N+1}_{+}} t^{1-2s} \pa_k\(\bg^{ij} \sqrt{|\bg|}\) \pa_i\(\chi_1 W_{\delta} + \whp_{\dsz}\) \pa_j \(\chi_1W_{\delta} + \whp_{\dsz}\) dxdt \\
&= - {\kappa_s \over 2} \int_{\mr^{N+1}_{+}} t^{1-2s} \pa_k\(\bg^{ij} \sqrt{|\bg|}\) \pa_i\(\chi_1 W_{\delta}\) \pa_j \(\chi_1W_{\delta}\) dxdt \\
&\quad + O \(\int_{\mr^{N+1}_{+}} t^{1-2s} \left|\nabla (\chi_1 W_{\delta})\right| \left|\nabla \whp_{\dsz}\right| |(x,t)| dxdt + \int_{\mr^{N+1}_{+}} t^{1-2s} \left|\nabla \whp_{\dsz}\right|^2 dx dt\)\\
&= - {\kappa_s \over 2} \int_{B_{r_0}^+} t^{1-2s} \pa_k\(\bg^{ij} \sqrt{|\bg|}\) \pa_i W_{\delta} \pa_j W_{\delta} dxdt + O\(\ep^{2\gamma}\) + \begin{cases}
o\(\ep^{2s\alpha}\) &\text{for } N > 4s,\\
o\(\ep^{\alpha}\) &\text{for } N > 2s+1.
\end{cases} \end{aligned} \end{equation}
Also Lemma \ref{lemma_metric_exp} implies that
\[\pa_k \sqrt{|\bg|} = - H_k t - \frac{1}{6} \(R_{kl} + R_{lk}\) x_l + O\(|(x,t)|^2\)\]
and
\[\pa_k h^{ij} = -\frac{1}{3} \(R^{i\phantom{n} \phantom{n}j}_{\phantom{i}kl\phantom{j}} + R^{i\phantom{n} \phantom{n}j}_{\phantom{i}lk\phantom{j}}\) x_l + h^{ij}_{, (N+1)k} t + O\(|(x,t)|^2\),\]
from which we obtain
\begin{multline*}
\pa_k\(\bg^{ij} \sqrt{|\bg|}\) = \pa_k \bg^{ij} \sqrt{|\bg|} + \bg^{ij} \pa_k \sqrt{|\bg|}\\
= - \left[ \frac{1}{3} \(R^{i\phantom{n} \phantom{n}j}_{\phantom{i}kl\phantom{j}} + R^{i\phantom{n} \phantom{n}j}_{\phantom{i}lk\phantom{j}}\) + \frac{1}{6} \delta^{ij} \(R_{kl} + R_{lk}\) \right] x_l
+ \(h^{ij}_{, (N+1)k} - \delta^{ij} H_k\) t + O\(|(x,t)|^2\).
\end{multline*}
Inserting this into \eqref{eq-m1-1} and then applying \eqref{eq_decay_2} as well as the relations $h^{ii}_{\phantom{ii},(N+1)k} = \pi^{ii}_{\phantom{ii},k} = 2H_k$ and
\[\int_{B_{r_0}^+} t^{1-2s} x_l \pa_i W_{\delta} \pa_j W_{\delta} dxdt = 0 \quad \text{(by the odd symmetry of $W_{\delta}$ in the $x_1, \cdots, x_N$ variables)},\]
we get
\begin{equation}\label{eq_I_11'}
\begin{aligned}
I_{11}' &= {\kappa_s \over 2} \(\delta^{ij} H_k - h^{ij}_{\phantom{ij}, (N+1)k} \) {\delta_{ij} \over N} \int_{\mr^{N+1}_{+}} t^{2-2s} |\nabla_x W_{\delta}|^2 dxdt \\
&\quad + {\kappa_s \over 2} \left[ \frac{1}{3} \(R^{i\phantom{n} \phantom{n}j}_{\phantom{i}kl\phantom{j}} + R^{i\phantom{n} \phantom{n}j}_{\phantom{i}lk\phantom{j}}\) + \frac{1}{6} \delta^{ij} \(R_{kl} + R_{lk}\) \right] \int_{B_{r_0}^+} t^{1-2s} x_l \pa_i W_{\delta} \pa_j W_{\delta} dxdt \\
&\quad + O\(\int_{B_{r_0}^+} t^{1-2s}|(x,t)|^2 |\nabla W_{\delta}|^2 dxdt\) + O\(\ep^{2\gamma}\) + \begin{cases}
o\(\ep^{2s\alpha}\) &\text{for } N > 4s,\\
o\(\ep^{\alpha}\) &\text{for } N > 2s+1,
\end{cases}\\
&= {\kappa_s \over 2} \({N-2 \over N}\) H_k \delta \int_{\mr^{N+1}_{+}} t^{2-2s} |\nabla_x W_1|^2 dxdt + O\(\ep^{2\gamma}\) + \begin{cases}
o\(\ep^{2s\alpha}\) &\text{for } N > 4s,\\
o\(\ep^{\alpha}\) &\text{for } N > 2s+1.
\end{cases} \end{aligned} \end{equation}
Next the term $I_{12}'$ is to be considered. In fact, one can observe that
\begin{equation}\label{eq_I_12'}
\begin{aligned}
I_{12}' &= {\kappa_s \over 2} \int_{B_{4r_0}^+} t^{1-2s} \pa_k\(\pa_t\(\chi_1 W_{\delta} + \whp_{\dsz}\)\)^2 \sqrt{|\bg|} dx dt \\
&= - {\kappa_s \over 2} \int_{B_{4r_0}^+} t^{1-2s} \(\pa_t\(\chi_1 W_{\delta} + \whp_{\dsz}\)\)^2 \pa_k \sqrt{|\bg|} dx dt \\
&= {\kappa_s \over 2} \int_{B_{r_0}^+} t^{1-2s} \(\pa_t W_{\delta}\)^2 \left[ H_k t + \frac{1}{6} \(R_{kl} + R_{lk}\) x_l + O\(|(x,t)|^2\)\right] dx dt \\
& \hspace{175pt} + O\(\ep^{2\gamma}\) + \begin{cases}
o\(\ep^{2s\alpha}\) &\text{for } N > 4s,\\
o\(\ep^{\alpha}\) &\text{for } N > 2s+1,
\end{cases}\\
&= {\kappa_s \over 2} H_k \int_{\mr^{N+1}_+} t^{2-2s} \(\pa_t W_{\delta}\)^2 dx dt + O\(\ep^{2\gamma}\) + \begin{cases}
o\(\ep^{2s\alpha}\) &\text{for } N > 4s,\\
o\(\ep^{\alpha}\) &\text{for } N > 2s+1.
\end{cases} \end{aligned} \end{equation}
Consequently, \eqref{eq_I_11'}, \eqref{eq_I_12'} and \eqref{eq_appen_11} give us that
\begin{equation}\label{eq_I_1'}
\begin{aligned}
I_1' &= {\kappa_s \over 2} H_k \lambda \ep^{\alpha} \int_{\mr^{N+1}_{+}} t^{2-2s} \left[\({N-2 \over N}\) |\nabla_x W_1|^2 + \(\pa_t W_1\)^2 \right] dxdt + o\(\ep^{\alpha}\) \\
&= \kappa_s H_k \lambda \ep^{\alpha} \({2N-2s-1 \over 4N}\) \int_{\mr^{N+1}_{+}} t^{2-2s} |\nabla W_1|^2 dxdt + o\(\ep^{\alpha}\).
\end{aligned} \end{equation}

\medskip \noindent
\textsc{2. Estimate of $I'_2$}.
Performing the integration by parts, we have
\[I'_2 = - {\kappa_s \over 2} \int_{\mr^{N+1}_{+}} \pa_k \(E(t)\sqrt{|\bg|}\) \(\chi_1 W_{\delta} + \whp_{\dsz}\)^2 dx dt + O\(\ep^{2\gamma}\).\]
If $s \in (1/2, 1)$ and $H = 0$, then
\begin{equation}\label{eq_I_2'}
\begin{aligned}
\left|I'_2\right| &= O\(\int_{\mr^{N+1}_{+}} t^{1-2s} | W_{\delta}|^2 dx dt\) + O\(\int_{\mr^{N+1}_{+}} t^{1-2s} \left|\whp_{\dsz}\right|^2 dx dt\) \\
&= O\(\ep^{2\alpha}\) + O\(\ep^{2\gamma}\) = \begin{cases}
o(\ep) &\text{for } \eqref{eq_main_12},\\
o\(\ep^{\alpha}\) &\text{for } \eqref{eq_main_21}.
\end{cases} \end{aligned} \end{equation}
If $s \in (0,1/2)$, one finds from \eqref{eq_E_local_2} that $|I_2'| = O(\delta) + O(\ep^{2\gamma}) = o(\ep)$ for equations \eqref{eq_main_12}.
Furthermore, if $N > 2s + 1$ is imposed, we can compute that
\begin{equation}\label{eq_I_2'_2}
\begin{aligned}
I'_2 &= -{\kappa_s \over 2} \({N-2s \over 2}\) \int_{\mr^{N+1}_+} \pa_k H(\sigma(x)) t^{-2s} W_{\delta}^2 dxdt + O\(\int_{\mr^{N+1}_+} t^{-2s} \(W_{\delta}^2 + \whp_{\dsz}^2\)|(x,t)|dxdt\)\\
&\quad + O\(\int_{\mr^{N+1}_+} t^{-2s} \(\left|\chi_1^2-1\right|W_{\delta}^2 + \chi_1W_{\delta} \left|\whp_{\dsz}\right| + \wtp_{\dsz}^2\)dxdt\)\\
&= - \kappa_s \({N-2s \over 4}\) \pa_k \left.\( H(\sigma(x))\)\right|_{x=0} \delta \int_{\mr^{N+1}_+} t^{-2s} W_1^2 dxdt + o(\delta)
\end{aligned}
\end{equation}
for equation \eqref{eq_main_21}, by utilizing \eqref{eq_error_est_32}.

\medskip \noindent
\textsc{3. Estimate of $I'_3$}.
We have
\begin{align*}
I'_3 &= -{1 \over 2} \int_{\mr^N} \pa_k\(\tf(\sigma(x))\sqrt{|\hh|}\) \(\chi_1w_{\delta} + \whp_{\dsz}\)^2 dx + O\(\ep^{2\gamma}\)\\
&= -{1 \over 2} \int_{\mr^N} \pa_k\(\tf(\sigma(x))\) w_{\delta}^2 dx + \int_{\mr^N} O(|x|) \left|\chi_1w_{\delta} + \whp_{\dsz}\right|^2 dx \\
&\quad -{1 \over 2} \int_{\mr^N} \pa_k\(\tf(\sigma(x))\) \((\chi_1^2-1)w_{\delta}^2 + 2\chi_1 w_{\delta} \whp_{\dsz} + \whp_{\dsz}^2\) \sqrt{|\hh|} dx + O\(\ep^{2\gamma}\) \\
&= -{1 \over 2} \int_{\mr^N} \pa_k\(\tf(\sigma(x))\) w_{\delta}^2 dx + O\(\int_{B^N(0,2r_0)} |W_{\delta}| \left|\whp_{\dsz}\right| dx\) + o\(\ep^{2s\alpha}\) + O\(\ep^{2\gamma}\).
\end{align*}
Applying H\"older's inequality, we estimate
\[\int_{B^N(0,2r_0)} w_{\delta} \left|\whp_{\dsz}\right| dx = O\(\|w_{\delta}\|_{L^2(\mr^N)}\|\Phi_{\dsz}\|_{L^2(M)}\) = O\(\ep^{s\alpha + \gamma}\).\]
Hence it follows that
\begin{equation}\label{eq_I_3'}
I'_3 = \begin{cases}
-\dfrac{1}{2} \lambda^{2s} \ep \left. \pa_k \(\tf(\sigma(x))\)\right|_{x=0} \int_{\mr^N} w_1^2 dx + o(\ep) &\text{for } \eqref{eq_main_12},\\
0 &\text{for } \eqref{eq_main_21}.
\end{cases} \end{equation}

\medskip \noindent
\textsc{4. Estimate of $I'_4$}. We will deal with the cases \eqref{eq_main_12} only.
The remaining case \eqref{eq_main_21} is similar, and especially, the small linear term $\ep fu$ of $g_{\ep}(u)$ for this problem (see \eqref{eq_gep}) can be taken into consideration as in the previous step. One has
\begin{align*}
I'_4 &= \int_{\mr^N} \pa_k G_{\ep}\(\chi_1 W_{\delta} + \whp_{\dsz}\) \sqrt{|\hh|} dx + O\(\ep^{2\gamma}\) = - \int_{\mr^N} G_{\ep}\(\chi_1 W_{\delta} + \whp_{\dsz}\) \pa_k \sqrt{|\hh|} dx + O\(\ep^{2\gamma}\)\\
&= - \int_{\mr^N} G_{\ep}\(\chi_1 W_{\delta}\) \pa_k \sqrt{|\hh|} dx + \int_{\mr^N} \left[ G_{\ep}\(\chi_1 W_{\delta}\) - G_{\ep}\(\chi_1 W_{\delta} + \whp_{\dsz}\) \right] \pa_k \sqrt{|\hh|} dx + o(\ep).
\end{align*}
With the observation that $\pa_k \sqrt{|\hh|} = - {1 \over 6} \(R_{kl} + R_{lk}\) x_l + O(|x|^2)$, we estimate the second term as
\begin{align*}
&\int_{\mr^N} \left| G_{\ep}\(\chi_1 W_{\delta}\) - G_{\ep}\(\chi_1 W_{\delta} + \whp_{\dsz}\) \right| O(|x|) dx
\\
&\le \int_{\mr^N} \( (\chi_1W_{\delta})^{p \pm \ep} \left|\whp_{\dsz}\right| + \left|\whp_{\dsz}\right|^{p+1 \pm \ep}\) O(|x|) dx
\\
&\le C \( \int_{B^N(0,2r_0)} W_{\delta}^{p+1 \pm \(\frac{2N}{N+2s}\)\ep} |x|^{\frac{2N}{N+2s}} dx\)^{\frac{N+2s}{2N}}  \(\int_{\mr^N} \left|\whp_{\dsz}\right|^{p+1} dx\)^{1 \over p+1}
+ C \(\int_{\mr^N} \left|\whp_{\dsz}\right|^{p+1 \pm \ep} dx\)
\\
&= O\(\ep^{\alpha + \gamma}\) 
+ O\(\ep^{(p+1)\gamma}\) = o(\ep).
\end{align*}
In addition, we find
\[\int_{\mr^N} G_{\ep}\(\chi_1 W_{\delta}\) \pa_k \sqrt{|\hh|} dx = O\( \int_{\mr^N}  W_{\delta}^{p+1 \pm \ep} |x|^2 dx \) = O\(\ep^{2\gamma}|\log\ep|\) = o(\ep)\]
given that $N \ge 2$.

Collecting \eqref{eq_I_1'}-\eqref{eq_I_3'} and the above estimates completes the proof of the $C^1$-estimates for $J_{\ep}$.
\end{proof}

\section{Conclusion of the proof of Theorems \ref{thm_main_1} and \ref{thm_main_2} and some remarks}
In this section, we complete the proof of our main results.
\begin{proof}[Proof of Theorem \ref{thm_main_1}]
Suppose that $\sigma_0 \in M$ is a $C^1$-stable critical point of $f$ such that $f(\sigma_0) > 0$.
If we let
\[\widetilde{J}(\ls) = {d_1 \over 2}f(\sigma)\lambda^{2s} - {(N-2s)^2d_0 \over 4N} \log\lambda \quad \text{for } (\ls) \in (0,\infty) \times M\]
and
\[\lambda_0 = \({(N-2s)^2d_0 \over 4Nf(\sigma_0)d_1}\)^{1 \over 2s} > 0,\]
then it follows from the invariance of the Brouwer degree under a homotopy that $(\lambda_0, \sigma_0)$ is a $C^1$-stable critical point of $\widetilde{J}$
(refer to the proof of Theorem 1.1 in \cite{MPV}).
Therefore, by Propositions \ref{prop_energy} and \ref{prop_energy_1},
there exists a critical point $(\lambda_{\ep}, \sigma_{\ep}) \in (0,\infty) \times M$ of $J_{\ep}$ in \eqref{eq_energy_est_res_1} for sufficiently small $\ep > 0$
such that $(\lambda_{\ep}, \sigma_{\ep}) \to (\lambda_0, \sigma_0)$ as $\ep \to 0$.
This fact and Proposition \ref{prop_red} imply that $(15_-)$ attains a positive solution.
As a consequence, we see from Proposition \ref{prop_ext} that its trace on $M$ solves $(1_-)$, deducing the conclusion.

If there is a $C^1$-stable critical point $\sigma_0 \in M$ of $f$ such that $f(\sigma_0) < 0$, then the same argument provides solutions of equations $(15_+)$, and so those of $(1_+)$.
This concludes the proof of Theorem $\ref{thm_main_1}$.
\end{proof}

\begin{proof}[Proof of Theorem \ref{thm_main_2}]
Under our assumptions existence of solutions to \eqref{eq_main_2} follows from Propositions \ref{prop_energy}, \ref{prop_energy_1}, \ref{prop_red} and \ref{prop_ext}.
Observe that $\lambda(\sigma)$ is the unique value such that ${\pa \widetilde{J} \over \pa \lambda}(\lambda(\sigma), \sigma) = 0$ for each $\sigma \in M$ fixed,
and $\nabla \widetilde{J}(\lambda(\sigma), \sigma) = 0$ if and only if $\sigma$ is a critical point of the function
\[\widehat{J}(\sigma) := \widetilde{J}(\lambda(\sigma), \sigma)
= \pm (d_1s)^{1 \over 1-2s} \(4N(1-2s) \over \(2N(N-1)+\(1-4s^2\)\)d_s\)^{2s \over 1-2s} \({1-2s \over 2s}\) \({|f(\sigma)| \over |H(\sigma)|^{2s}}\)^{1 \over 1-2s}.\]
Hence $\sigma_0$ should be a critical point of $|f|/|H|^{2s}$.
The proof is finished.
\end{proof}

We conclude this section, raising some additional questions regarding our main result.

First of all, one may ask the compactness issue for equations \eqref{eq_main_1} with $f = 0$.
For the local case ($s = 1$), if the dimension $N$ of a manifold $M$ satisfies $N \le 24$, the positive mass theorem holds for $M$ and the nonlinearity is slightly subcritical or critical,
then the solution set for \eqref{eq_main_1} is pre-compact as shown by Khuri, Marques and Schoen \cite{KMS}.
On the other hand, if $N \ge 7$ and the nonlinearity is slightly supercritical,
then Esposito and Pistoia \cite{EP} proved that there is a family of solutions to \eqref{eq_main_1} which blow-up at a maximum point of the function $x \to \|\text{Weyl}_{\hh}(x)\|_{\hh}$ defined for $x \in (M, \hh)$.
We think that a similar phenomenon may happen for the nonlocal case too, but do not have any definitive answer yet.

Secondly, the behavior of equation \eqref{eq_main_2} in the case $H = 0$ has to be understood.
Notice that the main order in the energy expansion \eqref{eq_energy_est_res_3}, computed with the assumption $H \ne 0$, is $\ep^{1 \over 1-2s}$ whose exponent is well-defined (namely, positive) only if $s \in (0, 1/2)$.
It would be interesting to figure out how this is related to the fact that the characterization of $P^s_{\hh}$ in terms of extension problems
is valid for any $H$ only if $s \in (0, 1/2]$, while the case $s = 1/2$ is quite special in that it arises from the purely local problem - the boundary Yamabe problem.
On the other hand, if $H = 0$, the correct choice of $\alpha$ in \eqref{eq_alpha}
and the main order of the energy expansion would be ${1 \over 2(1-s)}$ and $\ep^{1 \over 2(1-s)}$, respectively, hence it makes sense for any $s \in (0,1)$.
However, controlling this case is technically harder, since one needs to improve the accuracy of approximate solutions.
Such an additional difficulty also arose in the local cases ($s = 1, 2$) in \cite{EPV} and \cite{PV}.

In both problems, we suspect that the governing function for the blow-up location has a relationship with the norm of the second fundamental form $\|\pi\|_{\hh}$ or that of the Weyl tensor $\|\text{Weyl}_{\hh}\|_{\hh}$.
In \cite{GW, KMW}, one can observe how the Weyl tensor carries out its role in the fractional Yamabe problem.

Currently a theory for the higher order fractional Paneitz operator ($\gamma \in (1,2))$ is being developed (see e.g. \cite{CC}).
It seems natural to formulate analogous problems for these operators.
We also believe that equation $(1_{\pm})$ should have bubble-tower type solutions as in \cite{PVe}.

\bigskip \noindent \textbf{Acknowledgement.}
The authors would like to express their sincere gratitude to Professor A. Pistoia for her valuable comments.
W. Choi was supported by the Global Ph.D Fellowship of the Government of
South Korea 300-20130026.
Also, S. Kim has been supported by FONDECYT Grant 3140530, Chile.

\appendix
\section{Proof of Lemma \ref{lemma_decay}}\label{sec_appen}
In this appendix, we justify Lemma \ref{lemma_decay} which describes the decay of the bubble $W_{\delta}$.
The proof will be achieved once we combine Lemmas \ref{lem-appen-1} and \ref{lem-pw-decay}.

\begin{lemma}\label{lem-appen-1}
Let $0<s<1$ and $a \in \mr$. Also fix $0 < R_1 < R_2$ and denote $A^+_{\delta^{-1}} = B^+_{R_2\delta^{-1}} \setminus B^+_{R_1\delta^{-1}}$.
Then, as $\delta \to 0$, we have the estimates
\[\int_{A^+_{\delta^{-1}}} \frac{t^{1-2s}}{|(x,t)|^{N+2-2s+a}} dx dt =
\begin{cases}
O\(\delta^a\) &\text{for } a \ne 0,\\
O\(|\log \delta|\) &\text{for } a = 0,
\end{cases}\]
and
\[\int_{A^+_{\delta^{-1}}} \frac{t^{2s-1}}{|(x,t)|^{N+2s+a}} dx dt =
\begin{cases}
O\(\delta^a\) &\text{for } a \ne 0,\\
O\(|\log \delta|\) &\text{for } a = 0.
\end{cases}\]
\end{lemma}
\begin{proof}
The second inequality follows from the first inequality by substituting $s$ with $1-s$.
To prove the first inequality, we decompose the domain of integration
\[A^+_{\delta^{-1}} = \(A^+_{\delta^{-1}} \cup \{|t| \ge |x|\}\) \cup \(A^+_{\delta^{-1}} \cup \{|t| \le |x|\}\)\]
and estimate each part separately.
If $|t|\ge |x|$, then it holds that $|t| \le |(x,t)| \le \sqrt{2}|t|$. Hence we get
\begin{align*}
\int_{A^+_{\delta^{-1}} \cup \{|t| \ge |x|\}} \frac{t^{1-2s}}{|(x,t)|^{N+2-2s+a}} dx dt
&\le \max\left\{1, \sqrt{2}^{2s-1}\right\} \int_{A^+_{\delta^{-1}} \cup \{|t| \ge |x|\}} \frac{1}{|(x,t)|^{N+1+a}} dx dt
\\
&\le C \int_{A^+_{\delta^{-1}}} \frac{1}{|(x,t)|^{N+1+a}}dxdt =
\begin{cases}
O\(\delta^a\) &\text{for } a \ne 0,\\
O\(|\log \delta|\) &\text{for } a = 0.
\end{cases}
\end{align*}
If $|t|\le |x|$, then we have that ${\delta^{-1} \over \sqrt{2}} \le {1 \over \sqrt{2}}|(x,t)| \le |x| \le |(x,t)| \le 2\delta^{-1}$ for $(x,t) \in A^+_{\delta^{-1}}$. Consequently,
\begin{align*}
\int_{A^+_{\delta^{-1}} \cup \{|t| \le |x|\}} \frac{t^{1-2s}}{|(x,t)|^{N+2-2s+a}} dx dt
&\le \int_{\left\{\frac{\delta^{-1}}{\sqrt{2}} \le |x| \le 2\delta^{-1}\right\}} \int_{\{|t| \le |x|\}} \frac{t^{1-2s}}{|x|^{N+2-2s+a}} dt dx
\\
&= \frac{1}{1-s}\int_{\left\{\frac{\delta^{-1}}{\sqrt{2}} \le |x| \le 2\delta^{-1}\right\}} \frac{|x|^{2-2s}}{|x|^{N+2-2s+a}} dx
\\
& = \frac{1}{1-s}\int_{\left\{\frac{\delta^{-1}}{\sqrt{2}} \le |x| \le 2\delta^{-1}\right\}} \frac{1}{|x|^{N+a}} dx =
\begin{cases}
O\(\delta^a\) &\text{for } a \ne 0,\\
O\(|\log \delta|\) &\text{for } a = 0.
\end{cases}
\end{align*}
Combination of the above two estimates yields the desired inequality, concluding the proof.
\end{proof}

\begin{lemma}\label{lem-pw-decay}
Assume that $|(x,t)| \ge R_0$ for some fixed $R_0 > 0$ large.
Then we have the validity of
\begin{enumerate}
\item[(i)] $W_1(x,t) \le \frac{C}{|(x,t)|^{N-2s}}$\quad  and \quad $|\nabla W_1(x,t)| \le \frac{C}{|(x,t)|^{N-2s+1}}+ \frac{Ct^{2s-1}}{|(x,t)|^{N+2s}}.$
\item[(ii)] $|\pa_i W_1(x,t)| \le \frac{C}{|(x,t)|^{N-2s+1}}$ \quad and $\quad |\nabla \pa_i W_1(x,t)| \le \frac{C}{|(x,t)|^{N-2s+2}} + \frac{C t^{2s-1}}{|(x,t)|^{N+2s+1}}$ for $i=1,2\cdots, N$.
\item[(iii)] $|\pa_{\delta} W_1(x,t)|\le \frac{C}{|(x,t)|^{N-2s}}$ \quad and \quad $|\nabla \pa_{\delta} W_1(x,t)| \le \frac{C}{|(x,t)|^{N-2s+1}} + \frac{Ct^{2s-1}}{|(x,t)|^{N+2s}}$
\end{enumerate}
for some $C > 0$ determined by $N$, $s$ and $R_0$.
\end{lemma}
\begin{proof}
We initiate the proof with recalling Green's representation formula
\begin{equation}\label{eq_Green}
W_{\delta}(x,t) = \mathfrak{a}_{N,s} \int_{\mr^N} \frac{w_{\delta}(y)^{\frac{N+2s}{N-2s}}}{|(x-y,t)|^{N-2s}} dy
= \mathfrak{b}_{N,s} \int_{\mr^N} \(\frac{\delta}{\delta^2+|y|^2}\)^{N+2s \over 2} \frac{1}{|(x-y,t)|^{N-2s}} dy
\end{equation}
where $\mathfrak{a}_{N,s}$ and $\mathfrak{b}_{N,s}$ are positive constants depending only on $N$ and $s$ (see \cite[Subsection 2.3]{CKL}).
The proof consists of 3 steps.

\noindent \emph{Step 1: Estimates of $W_1$.} We split the situation into two cases.

\noindent Case 1. Assume that $|x|\le |t|$. Since $|(x,t)|\le \sqrt{2}|t|$, we obtain
\[W_1(x,t) \le \mathfrak{b}_{N,s} \int_{\mr^N} \frac{1}{(1+|y|^2)^{N+2s \over 2}} \frac{1}{|t|^{N-2s}} dy= \frac{C}{|t|^{N-2s}} \le \frac{C}{|(x,t)|^{N-2s}}.\]
Case 2. Assume next that $|x|\ge |t|$. Then we observe from $|(x,t)|\le \sqrt{2}|x|$ that
\[W_1(x,t) \le \mathfrak{b}_{N,s} \int_{\mr^N} \frac{1}{(1+|y|^2)^{N+2s \over 2}} \frac{1}{|x-y|^{N-2s}} dy = w_1(x) \le \frac{C}{|x|^{N-2s}} \le \frac{C}{|(x,t)|^{N-2s}}.\]
Putting these two estimates together, we get the first inequality of (i).

\medskip
\noindent \emph{Step 2: Estimates of $|\nabla W_1|$}. Again we deal with the two mutually exclusive cases.

\noindent Case 1. Suppose $|x|\le |t|$. Then, from we have $|(x,t)| \le \sqrt{2}|t|$, we see that
\begin{equation}\label{eq_appen_6}
\begin{aligned}
|\nabla_{(x,t)} W_1(x,t)| &\le \mathfrak{b}_{N,s} \int_{\mr^N} \frac{1}{(1+|y|^2)^{N+2s \over 2}} \left|\nabla_{(x,t)} \frac{1}{|(x-y,t)|^{N-2s}}\right| dy
\\
& \le C \int_{\mr^N} \frac{1}{(1+|y|^2)^{N+2s \over 2}} \frac{1}{|(x-y,t)|^{N-2s+1}} dy
\\
&\le C \int_{\mr^N} \frac{1}{(1+|y|^2)^{N+2s \over 2}} \frac{1}{|t|^{N-2s+1}} dy
= \frac{C}{|t|^{N-2s+1}} \le \frac{C}{|(x,t)|^{N-2s+1}}.
\end{aligned} \end{equation}
Case 2. Assume that $|x| \ge |t|$ so that we get $|(x,t)|\le \sqrt{2}|x|$.
By integration by parts, we deduce
\begin{align*}
\nabla_{x} W_1(x,t) &= - \mathfrak{b}_{N,s} \int_{\mr^N} \frac{1}{(1+|y|^2)^{N+2s \over 2}} \nabla_y \( \frac{1}{|(x-y,t)|^{N-2s}}\) dy
\\
&= -\int_{|y-x|\ge \frac{|x|}{2}} \frac{1}{(1+|y|^2)^{N+2s \over 2}} \nabla_y \( \frac{1}{|(x-y,t)|^{N-2s}}\) dy
\\
&\quad + \int_{|y-x|\le \frac{|x|}{2}} \nabla_y\( \frac{1}{(1+|y|^2)^{N+2s \over 2}}\) \frac{dy}{|(x-y,t)|^{N-2s}}
- \int_{|y-x|= \frac{|x|}{2}} \frac{1}{(1+|y|^2)^{N+2s \over 2}} \frac{\nu_y dS_y}{|(x-y,t)|^{N-2s}}
\end{align*}
where $\nu_y$ and $dS_y$ is the outward unit normal vector and the surface measure on the sphere $|y-x|= \frac{|x|}{2}$, respectively.
Hence, realizing that $|y| \ge \frac{|x|}{2}$ if $|y-x| \le \frac{|x|}{2}$, we derive from the above that
\begin{equation}\label{eq_appen_4}
\begin{aligned}
&|\nabla_x W_1(x,t)|\\
&\le \frac{C}{|x|^{N-2s+1}}\int_{|y-x| \ge \frac{|x|}{2}} \frac{1}{(1+|y|^2)^{N+2s \over 2}} dy
+ \frac{C}{|x|^{N+2s+1}}\int_{|y-x| \le \frac{|x|}{2}}  \frac{1}{|(x-y,t)|^{N-2s}} dy + O \(\frac{|x|^{N-1}}{|x|^{2N}}\)
\\
& = O\(\frac{1}{|x|^{N-2s+1}}\) + O\(\frac{1}{|x|^{N+2s+1}} \cdot |x|^{2s}\) + O\(\frac{|x|^{N-1}}{|x|^{2N}}\) \le \frac{ C}{|x|^{N-2s+1}} \le \frac{ C}{|(x,t)|^{N-2s+1}},
\end{aligned} \end{equation}
which with \eqref{eq_appen_6} implies the first inequality of (ii).

On the other hand, for $|x| \ge |t|$ and $|y-x|\ge \frac{|x|}{2}$, we have
\begin{equation}\label{eq_appen_2}
\begin{aligned}
\int_{|y-x| \ge \frac{|x|}{2}} \frac{1}{(1+|x-y|^2)^{N+2s \over 2}} \frac{t}{|(y,t)|^{N-2s+2}} dy & \le \frac{1}{|x|^{N+2s}} \int_{\mr^N} \frac{t}{|(y,t)|^{N-2s+2}} dy
\\
&= \frac{1}{|x|^{N+2s}} \int_{\mr^N} \frac{t \cdot t^N}{t^{N-2s+2} |(y,1)|^{N-2s+2}} dy
\\
& = \frac{C t^{2s-1}}{|x|^{N+2s}} \le \frac{C t^{2s-1}}{|(x,t)|^{N+2s}}.
\end{aligned} \end{equation}
Moreover for $|x| \ge |t|$ and $|y-x| \le \frac{|x|}{2}$, it holds that $|y| \ge \frac{|x|}{2}$, from which we find
\begin{equation}\label{eq_appen_3}
\begin{aligned}
\int_{|y-x| \le \frac{|x|}{2}} \frac{1}{(1+|x-y|^2)^{N+2s \over 2}} \frac{t}{|(y,t)|^{N-2s+2}} dy
&\le \frac{t}{|x|^{N-2s+2}}\int_{\mr^N}\frac{1}{(1+|x-y|^2)^{N+2s \over 2}} dy
\\
&=\frac{C t}{|x|^{N-2s+2}} \le \frac{Ct}{|(x,t)|^{N-2s+2}}\le \frac{C}{|(x,t)|^{N-2s+1}}.
\end{aligned} \end{equation}
As a result, thanks to \eqref{eq_appen_2} and \eqref{eq_appen_3}, we obtain
\begin{equation}\label{eq_appen_5}
|\pa_t W_1(x,t)| \le C \int \frac{1}{(1+|x-y|^2)^{N+2s \over 2}}\frac{t}{|(y,t)|^{N-2s+2}} dy \le  \frac{t^{2s-1}}{|(x,t)|^{N+2s}}+\frac{1}{|(x,t)|^{N-2s+1}}.
\end{equation}
Now \eqref{eq_appen_6}, \eqref{eq_appen_4} and \eqref{eq_appen_5} give us the second inequality of (i).

\medskip
\noindent \emph{Step 3: Estimates of $|\nabla \pa_i W_1|$, $|\pa_{\delta}W_1|$ and $|\nabla \pa_{\delta}W_1|$.}
Following the same procedure which was applied to $W_1$ and $\nabla W_1$ in Steps 1 and 2, one can find an upper bound of $|\nabla \pa_i W_1|$ for each $i = 1, \cdots, N$, and in particular the second inequality of (ii).

Meanwhile, we discover from \eqref{eq_Green} that
\[\pa_{\delta}W_1(x,t) = \mathfrak{b}_{N,s} \({N+2s \over 2}\) \int_{\mr^N} \frac{|y|^2-1}{(1+ |y|^2)^{\frac{N+2s}{2}+1}} \frac{1}{|(x-y,t)|^{N-2s}} dy.\]
Because
\[\frac{\left||y|^2-1\right|}{(1+|y|^2)^{\frac{N+2s}{2}+1}} \le \frac{1}{(1+|y|^2)^{N+2s \over 2}} \quad \text{for any } y \in \mr^N,\]
we can get (iii) by adopting the argument in Steps 1 and 2 once more.
This completes the proof.
\end{proof}

\Addresses
\end{document}